\documentclass[a4paper, intlimits, reqno]{amsart}

\usepackage[english]{babel}
\usepackage[T1]{fontenc} 
\usepackage[utf8]{inputenc}

\usepackage{amsmath}
\usepackage{amssymb}
\usepackage{MnSymbol}
\usepackage{amsthm}
\allowdisplaybreaks
\usepackage{amsfonts}
\usepackage{mathrsfs} 
\usepackage{mathtools}
\usepackage{nicefrac}
\usepackage{enumitem}
\usepackage{multicol}
\usepackage{url}
\usepackage{dsfont}
\usepackage[numbers,sort&compress]{natbib}
\usepackage{caption}
\usepackage{tikz}
\usetikzlibrary{patterns,calc,arrows}
\tikzstyle{arrow}=[draw, -latex]
\usepackage{prettyref}
\usepackage{doi}
\usepackage{orcidlink}

\newrefformat{def}{Definition \ref{#1}}
\newrefformat{rem}{Remark \ref{#1}}
\newrefformat{sect}{Section \ref{#1}}
\newrefformat{prop}{Proposition \ref{#1}}
\newrefformat{thm}{Theorem \ref{#1}}
\newrefformat{cor}{Corollary \ref{#1}}
\newrefformat{ex}{Example \ref{#1}}
\newrefformat{fig}{Figure \ref{#1}}

\swapnumbers
\newtheoremstyle{dotless}{}{}{\itshape}{}{\bfseries}{}{}{}
\theoremstyle{dotless}
\theoremstyle{plain}
\newtheorem{thm}{Theorem}[section]

\theoremstyle{definition}
\newtheorem{defn}[thm]{Definition}
\newtheorem{rem}[thm]{Remark}

\newcommand{\N} {\mathbb{N}}
\newcommand{\R} {\mathbb{R}}
\newcommand{\C} {\mathbb{C}}

\DeclareMathOperator{\re}{Re}
\DeclareMathOperator{\im}{Im}
\providecommand{\differential}{\mathrm{d}}
\renewcommand{\d}{\differential}

\newcommand{\vertiii}[1]{{\left\vert\kern-0.25ex\left\vert\kern-0.25ex\left\vert #1 
    \right\vert\kern-0.25ex\right\vert\kern-0.25ex\right\vert}}  

\begin{document}

\title[Abstract Cauchy problem]{The abstract Cauchy problem in locally convex spaces}
\author[K.~Kruse]{Karsten Kruse\,\orcidlink{0000-0003-1864-4915}}
\thanks{K.~Kruse acknowledges the support by the Deutsche Forschungsgemeinschaft (DFG) within the Research Training
 Group GRK 2583 ``Modeling, Simulation and Optimization of Fluid Dynamic Applications''.}
\address{Hamburg University of Technology\\ Institute of Mathematics \\
Am Schwarzenberg-Campus~3 \\
21073 Hamburg \\
Germany}
\email{karsten.kruse@tuhh.de}



\date{\today}
\begin{abstract}
We derive necessary and sufficient criteria for the uniqueness and existence of solutions 
of the abstract Cauchy problem in locally convex Hausdorff spaces. 
Our approach is based on a suitable notion of an asymptotic Laplace transform 
and extends results of Langenbruch beyond the class of Fr\'echet spaces. 
\end{abstract}

\maketitle
\section{Introduction}
\label{sect:Introduction}

We study the abstract Cauchy problem in locally convex Hausdorff spaces in the present paper. 
This is an initial value problem of the form
\begin{align*}
x'(t)&=Ax(t),\quad t>0,\\
x(0)&=x_{0}\in E,
\end{align*}
where $A\colon D(A)\subset E\to E$ is a sequentially closed linear operator and $E$ a sequentially complete 
locally convex Hausdorff space over $\C$. 

One of the approaches to tackle the abstract Cauchy problem is the theory of $C_{0}$-semigroups. 
The classical theory of $C_{0}$-semigroups on Banach spaces (see e.g.\ \cite{engel_nagel2000} 
and the references therein) has already been extended in several ways.
Beyond the realm of Banach spaces it was extended to equicontinuous $C_{0}$-semigroups on locally convex 
Hausdorff spaces in \cite[Chap.~IX]{albanese2016,komatsu1964,yosida1968}, 
quasi-equicontinuous $C_{0}$-semigroups 
in \cite{albanese2016,babalola1974,choe1985,jacob2015_2,miyadera1959,moore1971a,moore1971b}, 
locally equicontinuous $C_{0}$-semigroups in \cite{dembart1974,komura1968,Ouchi2},  
sequentially (locally) equicontinuous $C_{0}$-semigroups in \cite{federico2020}
and smooth semigroups on convenient algebras in \cite{teichmann2002}.

Besides the extension of the theory of $C_{0}$-semigroups to locally convex Hausdorff spaces the 
continuity assumptions were weakened as well. Bi-continuous semigroups were introduced 
in \cite{kuehnemund2001,kuehnemund2003}, (locally equi-)tight bi-continuous semigroups in \cite{es_sarhir2006}, 
integrable semigroups considered in \cite{kunze2009}, integrated semigroups in \cite{wang1997,xiao1998}, 
distribution semigroups on Banach spaces in \cite{chazarain1971,kunstmann1991,lions1960,melnikova2001,kisynski2007} 
and even on locally convex Hausdorff spaces in \cite{shiraishi1963,shiraishi1964}, and 
(Fourier) hyperfunction semigroups for Banach spaces in \cite{Ouchi1971,Ouchi1} (and \cite{ito1982a,ito1982b}).
We note that most of the classical bi-continuous semigroups are actually quasi-equicontinuous 
$C_{0}$-semigroups with respect to the mixed topology and even quasi-equitight 
by \cite[Theorem 7.4, p.\ 180]{kraaij2016} and \cite[3.17 Theorem, p.\ 13, Section 4]{kruse_schwenninger2022}.

Apart from the theory of semigroups some of the classical methods for initial value problems were transferred 
to the setting in locally convex Hausdorff spaces in \cite{lobanov1994} and the references therein. 

A common problem in the mentioned approaches to the abstract Cauchy problem is the development of 
a suitable notion of a Laplace transform for vector-valued generalised functions. 
In \cite{baeumer1997,baeumer2001,lumer1999} an appropriate (asymptotic) Laplace transform was developed 
for Banach-valued locally integrable functions and applied to abstract Cauchy problems, and 
in \cite{komatsu1988,Kom3,Kom4,Kom5,Kom6} under minimal regularity assumptions 
for Banach-valued hyperfunctions as well as in \cite{langenbruch2011} for Fr\'echet-valued hyperfunctions. 
This was extended in \cite{kruse2021_1} beyond the class of Fr\'echet spaces to a large variety of locally 
convex Hausdorff spaces containing common spaces of distributions. 

We use the asymptotic Laplace transform from \cite{kruse2021_1} to study the abstract Cauchy problem 
for vector-valued hyperfunctions. After recalling the necessary notions and results from \cite{kruse2021_1} 
in \prettyref{sect:Notation}, 
we characterise the uniqueness of the solutions of the abstract Cauchy problem in \prettyref{sect:uniqueness}, 
in particular, we derive necessary and sufficient conditions in \prettyref{thm:unique_prop_equiv}. 
We use these conditions to phrase sufficient conditions for the uniqueness of solutions in terms of asymptotic 
left resolvents in \prettyref{thm:unique_prop_suff_0} and \prettyref{thm:unique_prop_suff}, generalising 
corresponding results from \cite{langenbruch2011} (for general notions of resolvents in locally convex Hausdorff 
spaces see \cite{arikan2003} and \cite{domanskilangenbruch2012}).
In \prettyref{sect:solvability} we turn to the solvability of the abstract Cauchy problem 
for vector-valued hyperfunctions and present necessary and sufficient conditions for the solvability in 
\prettyref{thm:solution_ACP}. Here we use the Laplace transform for vector-valued Laplace hyperfunctions 
from \cite{domanskilangenbruch2010} in combination with the asymptotic Laplace transform for vector-valued 
hyperfunctions from \cite{kruse2021_1}. In \prettyref{thm:solution_ACP_sufficient} we give a sufficient condition 
for the solvability of the abstract Cauchy problem in terms of asymptotic right resolvents. 
Our results on the solvability extend the ones from \cite{langenbruch2011}.

\section{Notation and Preliminaries}
\label{sect:Notation}

We use essentially the same notation and preliminaries as in \cite[Section 2]{kruse2021_1}.
In the following $E$ is always a locally convex Hausdorff space over $\C$ equipped with a directed 
system of seminorms $(p_{\alpha})_{\alpha\in\mathfrak{A}}$, in short, $E$ is a $\C$-lcHs.
If $E$ is a normed space, we often write $\|\cdot\|_{E}$ for the norm on $E$.
We denote by $L(F,E)$ the space of continuous linear maps from 
a $\C$-lcHs $F$ to $E$, write $L(F)\coloneq L(F,F)$, sometimes use the notion $\langle T, f\rangle\coloneq T(f)$, $f\in F$, for 
$T\in L(F,E)$ and the symbol $T^{\operatorname{t}}$ for the dual map of $T$. 
If $E=\C$, we write $F'\coloneq L(F,\C)$ for the dual space of $F$. 
We denote by $L_{b}(F,E)$ the space $L(F,E)$ equipped with the locally convex topology of uniform convergence 
on the bounded subsets of $F$. 

We denote by $\mathcal{O}(\Omega,E)$ the space of $E$-valued holomorphic functions on an open 
set $\Omega\subset\C$ and by $\mathcal{C}^{\infty}(\Omega,E)$ the space of $E$-valued infinitely continuously 
partially differentiable functions on an open set $\Omega\subset\R^{2}=\C$. 
We denote by $\partial^{\beta}f$ the partial derivative of $f\in \mathcal{C}^{\infty}(\Omega, E)$ for a multiindex 
$\beta\in\N_{0}^{2}$.
We denote by $\C_{\re >0}\coloneq\{z\in\C\;|\;\re(z)>0\}$ the right halfplane, 
by $\overline{\R}\coloneq\R\cup\{\pm \infty\}$ the two-point compactification of $\R$ and set 
$\overline{\C}\coloneq\overline{\R}+i\R$. We define the distance of $z\in\C$ to a set $M\subset\C$ w.r.t.\ the Euclidean 
norm $|\cdot|$ via $\d(z,M)\coloneq\inf_{w\in M}|z-w|$ if $M\neq\emptyset$, and $\d(z,M)\coloneq\infty$ if $M=\emptyset$.
For a compact set $K\subset \overline{\R}$ and $c>0$ we define the sets
\begin{align*}
  U_{\frac{1}{c}}(K) 
&\coloneq\phantom{\cup} \{ z \in \C \:| \: \d(z,K\cap\C)< c\} \\
&\phantom{\coloneq}\cup
\begin{cases}
\emptyset &,\; K \subset \R, \\
]\nicefrac{1}{c},\infty[+i ]-c, c[ &,\; \infty \in K, \, -\infty \notin K, \\
]-\infty, -\nicefrac{1}{c}[+i ]-c, c[ &, \; \infty \notin K, \,  -\infty \in K,\\ 
\left(]-\infty, -\nicefrac{1}{c}[\cup ]\nicefrac{1}{c},\infty[\right)+i ]-c, c[ &,\; \infty \in K, \, -\infty \in K,
\end{cases}
\end{align*}
and for $n\in\N$ 
\[
S_{n}(K)\coloneq\left(\C\setminus\overline{U_{n}(K)}\right) \cap \{z\in\C\;|\;|\im(z)|<n\}. 
\]

\begin{defn}[{\cite[3.2 Definition, p.\ 12--13]{ich}}]
Let $E$ be a $\C$-lcHs and $K\subset\overline{{\R}}$ be compact.
\begin{enumerate}
\item [(a)] The space of vector-valued slowly increasing infinitely continuously partially differentiable 
functions outside $K$ is defined as
\[
\mathcal{E}^{exp}(\overline{\C}\setminus K, E)\coloneq \{f\in\mathcal{C}^{\infty}(\C\setminus K,E)
\; | \; \forall\;n\in\N,\,m\in\N_{0},\,\alpha\in\mathfrak{A}:\;\vertiii{f}_{n,m,\alpha,K} < \infty\}
\]
where
\[
\vertiii{f}_{n,m, \alpha,K}\coloneq\sup_{\substack{z\in S_{n}(K)\\ \beta\in\N_{0}^{2}, |\beta|\leq m}}
p_{\alpha}(\partial^{\beta}f(z))
e^{-\frac{1}{n}|\re(z)|}.
\]
 \item [(b)] The space of vector-valued slowly increasing holomorphic functions outside $K$ is defined as
\[
\mathcal{O}^{exp}(\overline{\C}\setminus K,E)\coloneq \{f\in\mathcal{O}(\C\setminus K,E)
\; | \; \forall\;n\in\N,\,\alpha\in\mathfrak{A}:\; \vertiii{f}_{n,\alpha,K} < \infty \}
\]
where
\[
\vertiii{f}_{n, \alpha,K}\coloneq\sup_{z \in S_{n}(K)}p_{\alpha}(f(z))e^{-\frac{1}{n}|\re(z)|}.
\]
Furthermore, we set
\[
bv_{K}(E)\coloneq\mathcal{O}^{exp}(\overline{\C}\setminus K, E)/\mathcal{O}^{exp}(\overline{\C}, E).
\]
\end{enumerate}
\end{defn}

We note that $S_{1}(\overline{\R})=\emptyset$ and 
$\vertiii{f}_{1,m, \alpha,\overline{\R}}=-\infty=\vertiii{f}_{1, \alpha,\overline{\R}}$ 
for any $f\colon\C\setminus\R\to E$, $m\in\N_{0}$ and $\alpha\in\mathfrak{A}$.
Other common symbols for the spaces $\mathcal{E}^{exp}(\overline{\C}\setminus K,E)$ resp.\ 
$\mathcal{O}^{exp}(\overline{\C}\setminus K,E)$ are $\widetilde{\mathcal{E}}(\overline{\C}\setminus K,E)$ resp.\ 
$\widetilde{\mathcal{O}}(\overline{\C}\setminus K,E)$ (see \cite[1.2 Definition, p.\ 5]{J}). 

\begin{defn}[{(strictly) admissible, \cite[p.\ 55]{ich}}]
A $\C$-lcHs $E$ is called \emph{admissible}, if the Cauchy-Riemann operator
\[
\overline{\partial}\colon \mathcal{E}^{exp}(\overline{\C}\setminus K,E)\to \mathcal{E}^{exp}(\overline{\C}\setminus K,E)
\]
is surjective for any compact set $K\subset\overline{\R}$. $E$ is called \emph{strictly admissible} 
if $E$ is admissible and if, in addition,
\[
\overline{\partial}\colon \mathcal{C}^{\infty}(\Omega,E)\to\mathcal{C}^{\infty}(\Omega,E)
\]
is surjective for any open set $\Omega\subset\C$.
\end{defn}

If $E$ is strictly admissible and sequentially complete, then the sheaf of $E$-valued Fourier 
hyperfunctions is flabby and can be represented by boundary values of exponentially slowly increasing 
holomorphic functions (see \cite[Theorem 5.9, p.\ 33]{kruse2019_5}). 
In particular, its subsheaf of $E$-valued hyperfunctions 
is flabby under this condition as well. Moreover, we may regard $bv_{K}(E)$ as the space of 
$E$-valued Fourier hyperfunctions with support in $K\subset\overline{\R}$ under this condition 
by \cite[5.11 Lemma, p.\ 44]{kruse2019_5}.

\begin{thm}[{\cite[5.25 Theorem, p.\ 98]{ich}}]\label{thm:examples_strictly_admiss}
If
\begin{enumerate}
\item [(a)] $E$ is a $\C$-Fr\'echet space, or if
\item [(b)] $E\coloneq F_{b}'$ where $F$ is a $\C$-Fr\'echet space satisfying $(DN)$, or if
\item [(c)] $E$ is a complex ultrabornological PLS-space satisfying $(PA)$, 
\end{enumerate}
then $E$ is strictly admissible.
\end{thm}

The definitions of the topological invariants $(DN)$ and $(PA)$ are given in 
\cite[Chap.\ 29, Definition, p.\ 359]{meisevogt1997} and \cite[Section 4, Eq.\ (24), p.\ 577]{Dom1}, respectively. 
Besides every $\C$-Fr\'echet space, the theorem above covers the space $E=\mathcal{S}(\R^{d})_{b}'$ 
of tempered distributions, the space $\mathcal{D}(V)_{b}'$ of distributions 
and the space $\mathcal{D}_{(\omega)}(V)_{b}'$ of ultradistributions of Beurling type 
and many more spaces given in \cite{Dom1}, \cite[Corollary 4.8, p.\ 1116]{D/L}, 
\cite[Example 4.4, p.\ 14--15]{kruse2019_5} and \cite{vogt1983}.

\begin{defn}[{\cite[7.1 Definition, p.\ 106]{kruse2021_1}}]
Let $E$ be a $\C$-lcHs and $a\in\{0,\infty\}$. We define the space
\[
\mathcal{LO}_{[a,\infty]}(E)\coloneq\{f\in\mathcal{O}(\C_{\re >0},E)\;|\;\forall\;k\in\N,\,\alpha\in\mathfrak{A}:\;
\|f\|_{k,\alpha,[a,\infty]}<\infty\}
\]
where
\[
 \|f\|_{k,\alpha,[0,\infty]}\coloneq\sup_{\re(z)\geq \frac{1}{k}}p_{\alpha}(f(z))e^{-\frac{1}{k}|z|}
\]
resp. 
\[
 \|f\|_{k,\alpha,\{\infty\}}\coloneq\sup_{\re(z)\geq \frac{1}{k}}p_{\alpha}(f(z))e^{-\frac{1}{k}|z|+k\re(z)}.
\]
We omit the index $\alpha$ of the seminorms if $E$ is a normed space, and write 
$\mathcal{LO}_{[a,\infty]}\coloneq\mathcal{LO}_{[a,\infty]}(\C)$.
\end{defn}

Let $E$ be a sequentially complete $\C$-lcHs, $K\coloneq [0,\infty]$ or $K\coloneq\{\infty\}$ 
and equip $bv_{K}(E)$ with its usual quotient topology, which is Hausdorff locally 
convex by \cite[Remark 14, p\ 22]{kruse2019_2}. 
By \cite[Theorem 7.2 (ii), p.\ 106]{kruse2021_1} the Laplace transform 
\[
 \mathcal{L}\colon bv_{K}(E) \to \mathcal{LO}_{K}(E),\;
 \mathcal{L}([F])(\zeta)\coloneq\int_{\gamma_{K}}{F(z)e^{-z\zeta}\d z},
\]
where $\gamma_{K}$ is the path along the boundary of $U_{\nicefrac{1}{c}}(K)$ with clockwise orientation (see \prettyref{fig:path_int}), does not depend on the choice of $c>0$ and is a topological isomorphism.
\begin{center}
\begin{minipage}{\linewidth}
\centering
\begin{tikzpicture}
\draw[very thick]  (0,1.5) arc (90:270:15mm);
\draw[very thick] (2,1.5) -- (4,1.5) node [above=0.1pt] {$\gamma_{[0,\infty]}$};
\draw[->,very thick] (0,1.5) -- (2.2,1.5);
\draw[very thick] (2,-1.5) -- (0,-1.5);
\draw[->,very thick] (4,-1.5) -- (1.8,-1.5);
\draw(-1.5 cm,1pt)--(-1.5cm,-1pt)node[anchor=north east] {$-c$};
\draw(1pt,1.5 cm)--(-1pt,1.5 cm)node[anchor=south east] {$c$};
\draw(1pt,-1.5 cm)--(-1pt,-1.5 cm)node[anchor=north east] {$-c$};
\draw[->] (-4,0) -- (4,0) node[right] {$\re(z)$} coordinate (x axis);
\draw[dashed,very thick] (0,0)--(4,0);
\draw[->] (0,-2) -- (0,2) node[above] {$\im(z)$} coordinate (y axis);
\end{tikzpicture}
\end{minipage}
\captionsetup{type=figure}
\caption{Path $\gamma_{[0,\infty]}$ with $c>0$ (cf.\ \cite[Figure 1.2, p.\ 62]{kruse2021_1})}
\label{fig:path_int}
\end{center}
Let $E$ be an admissible $\C$-lcHs. Then the canonical 
(restriction) map 
\[
\mathcal{R}_{[0,\infty[}\colon \mathcal{O}^{exp}(\overline{\C}\setminus [0,\infty],E)/
\mathcal{O}^{exp}(\overline{\C}\setminus\{\infty\},E)
\to \mathcal{B}([0,\infty[,E),\;[F]\mapsto [F],
\]
is a linear isomorphism by \cite[Theorem 5.1, p.\ 96]{kruse2021_1} where 
\[
\mathcal{B}([0,\infty[,E)\coloneq\mathcal{O}(\C\setminus [0,\infty[,E)/\mathcal{O}(\C,E)
\]
is the space of hyperfunctions with values in $E$ and support in $[0,\infty[$. 
The combination of both results leads to the following theorem.

\begin{thm}[{\cite[7.4 Theorem, p.\ 106]{kruse2021_1}}]\label{thm:asymp_laplace}
Let $E$ be an admissible sequentially complete $\C$-lcHs.
Then the asymptotic Laplace transform 
\[
\mathcal{L}^{\mathcal{B}}\colon 
\mathcal{B}([0,\infty[,E)\to 
\mathcal{LO}_{[0,\infty]}(E)/\mathcal{LO}_{\{\infty\}}(E),\;
\mathcal{L}^{\mathcal{B}}(f)\coloneq [(\mathcal{L}\circ \mathcal{R}_{[0,\infty[}^{-1})(f)],
\]
is a linear isomorphism.
\end{thm}

\section{Uniqueness of solutions of the ACP}
\label{sect:uniqueness}

In this section we apply our results on the asymptotic Laplace transform of hyperfunctions 
with support in $[0,\infty[$ to the abstract Cauchy problem for hyperfunctions with values in an admissible 
(sequentially) complete $\C$-lcHs.
We start with a generalisation of an abstract Cauchy problem for hyperfunctions given in 
\cite[p.\ 60--61]{langenbruch2011}.
Let $(E,(p_{\alpha})_{\alpha\in\mathfrak{A}})$ be a sequentially complete $\C$-lcHs. We call 
\begin{equation}\label{eq:ACP}
x'(t)=Ax(t),\quad t>0,\quad
x(0)=x_{0}\in E,
\end{equation}
an \emph{abstract Cauchy problem (ACP)} where 
\[
A\colon F\coloneq D(A)\subset E\to E
\]
is a sequentially closed linear operator with domain $F\coloneq D(A)$. 
Then $F$ is a sequentially complete $\C$-lcHs when equipped 
with the graph topology $\tau_{A}$ given by the seminorms 
$(p_{\alpha,A}\coloneq p_{\alpha}+p_{\alpha}(A\cdot))_{\alpha\in\mathfrak{A}}$, and $A\colon F_{A}\coloneq (F,\tau_{A})\to E$ 
is continuous. 

\begin{rem}
\begin{enumerate}
\item[(a)] If $E$ is a $\C$-Fr\'echet space and $A$ and $F$ are as above, then $F_{A}=(F,\tau_{A})$ 
is also a Fr\'echet space and thus (strictly) admissibile by \prettyref{thm:examples_strictly_admiss}.
\item[(b)] If $E$ is a (strictly) admissible space, $F=E$ and $A\colon F\to E$ continuous, 
then $F_{A}=E$ as locally convex spaces and so $F_{A}$ is (strictly) admissible. 
\end{enumerate}
\end{rem}

An $F$-valued hyperfunction $[u]\in\mathcal{B}([0,\infty[,F_{A})$ is called a \emph{solution of the ACP} \eqref{eq:ACP} 
(in the sense of hyperfunctions) if 
\begin{equation}\label{eq:ACP_hyp}
\frac{d}{dt}[u]-A[u]=x_{0}\otimes\delta_{0}
\end{equation}
where $\delta_{0}\coloneq [z\mapsto -\tfrac{1}{2\pi i z}]$ is the Dirac hyperfunction 
(see \cite[4.11 Example, p.\ 96]{kruse2021_1}), 
$x_{0}\otimes\delta_{0}\coloneq x_{0}\delta_{0}$, $\tfrac{d}{dt}[u]\coloneq [\tfrac{d}{dz}u]$ and $A[u]\coloneq [z\mapsto Au(z)]$. 

We say that the ACP \eqref{eq:ACP} has the \emph{uniqueness property} (in the sense of hyperfunctions) 
if $[u]=0$ is the only solution of \eqref{eq:ACP_hyp} for $x_{0}=0$. 
Our next theorem generalises \cite[Theorem 7.1, p.\ 61]{langenbruch2011} and we note that its proof 
essentially remains the same.

\begin{thm}\label{thm:unique_prop_equiv}
Let $E$ be an admissible sequentially complete $\C$-lcHs and $A\colon F\coloneq D(A)\subset E\to E$ a sequentially closed 
linear operator with admissible $F_{A}=(F,\tau_{A})$. Then the following statements are equivalent:
\begin{enumerate}
\item[(a)] The ACP \eqref{eq:ACP} has the uniqueness property (in the sense of hyperfunctions).
\item[(b)] If $h\in\mathcal{LO}_{[0,\infty]}(F_{A})$ and $(z-A)h\in\mathcal{LO}_{\{\infty\}}(E)$, 
then $h\in\mathcal{LO}_{\{\infty\}}(F_{A})$.
\item[(c)] If $h\in\mathcal{LO}_{[0,\infty]}(F_{A})$ and $(z-A)h\in\mathcal{LO}_{\{\infty\}}(E)$, 
then $\{h(t)e^{nt}\;|\;t\geq \varepsilon\}$ is weakly bounded in $F_{A}$ for any $n\in\N$ and any (some) $\varepsilon>0$.
\end{enumerate}
\end{thm}
\begin{proof}
(a) $\Rightarrow$ (b): Due to \prettyref{thm:asymp_laplace} there is 
$[u]\in\mathcal{B}([0,\infty[,F_{A})$ such that 
$[h]=\mathcal{L}^{\mathcal{B}}([u])\in\mathcal{LO}_{[0,\infty]}(F_{A})/\mathcal{LO}_{\{\infty\}}(F_{A})$. 
By \cite[7.10 Proposition, p.\ 108]{kruse2021_1} and our assumption we have
\[
\mathcal{L}^{\mathcal{B}}\Bigl(\frac{d}{dt}[u]-A[u]\Bigr)=(z-A)\mathcal{L}^{\mathcal{B}}([u])=(z-A)[h]=0.
\]
From \prettyref{thm:asymp_laplace} and (a) we deduce that $[u]=0$, which implies $[h]=0$ and thus 
$h\in\mathcal{LO}_{\{\infty\}}(F_{A})$.

(b) $\Rightarrow$ (c): This follows from $h\in\mathcal{LO}_{\{\infty\}}(F_{A})$ by (b) and the definition of 
the space $\mathcal{LO}_{\{\infty\}}(F_{A})$.

(c) $\Rightarrow$ (a): Let $[u]\in\mathcal{B}([0,\infty[,F_{A})$ such that $\tfrac{d}{dt}[u]-A[u]=0$. 
Then $[h]\coloneq\mathcal{L}^{\mathcal{B}}([u])$ satisfies 
\[
0=\mathcal{L}^{\mathcal{B}}\Bigl(\frac{d}{dt}[u]-A[u]\Bigr)=(z-A)[h]
\]
and thus $(z-A)h\in\mathcal{LO}_{\{\infty\}}(E)$. 
Next, we show that $y \circ h\in\mathcal{LO}_{\{\infty\}}$ for any $y\in F_{A}'$
by the Phragm\'en--Lindel\"of theorem. Let $k\in\N$ and set 
$S\coloneq\{z\in\C\;|\;-\tfrac{\pi}{4}<\operatorname{arg}(z)<\tfrac{\pi}{4}\}$ 
and $S_{0}\coloneq\{z\in\C\;|\;\re(z)>\tfrac{1}{k},\,\im(z)>0\}$. 
We define the homeomorphism $\theta\colon \overline{S}\to\overline{S}_{0}$ by 
$\theta(z)\coloneq e^{i\tfrac{\pi}{4}}z+\tfrac{1}{k}$ and the function
\[
F_{0}\colon \overline{S}\to \C,\;F_{0}(z)\coloneq (y\circ h)(\theta(z))e^{(k+\frac{i}{k})\theta(z)}.
\] 
Using $y\circ h\in\mathcal{LO}_{[0,\infty]}$, we have for every $z\in\overline{S}$ that
\begin{align*}
 |F_{0}(z)|
&=|(y\circ h)(\theta(z))|e^{k\re(\theta(z))-\frac{1}{k}\im(\theta(z))}\\
&\leq |(y\circ h)(\theta(z))|e^{-\frac{1}{k}|\theta(z)|+(k+\frac{1}{k})|\re(\theta(z))|}\\
&\leq |y\circ h|_{k,[0,\infty]}e^{(k+\frac{1}{k})|\theta(z)|}
 \leq e^{1+\frac{1}{k^{2}}}|y\circ h|_{k,[0,\infty]}e^{(k+\frac{1}{k})|z|}.
\end{align*}
If $\arg(z)=-\tfrac{\pi}{4}$, then $\theta(z)=|z|+\tfrac{1}{k}$ and by part (c) there is $\varepsilon>0$ such that 
with $\varepsilon_{k}\coloneq\max(0,\varepsilon-\frac{1}{k})$ we get
\begin{flalign*}
&\hspace{0.35cm} |F_{0}(z)|\\
&=|(y\circ h)(\theta(z))|e^{k\theta(z)}\\
&\leq e^{k\varepsilon_{k}+1}\max_{\substack{\arg(w)=-\frac{\pi}{4} \\ |w|\leq \varepsilon_{k}}}|(y\circ h)(\theta(w))|
 +\sup_{\substack{\arg(w)=-\frac{\pi}{4} \\ |w|\geq \varepsilon_{k}}}|(y\circ h)(\theta(w))|e^{k\theta(w)}=:C_{0}<\infty
\end{flalign*}
where we use the continuity of $y\circ h\circ\theta$ as well. 
If $\arg(z)=\tfrac{\pi}{4}$, then $\theta(z)=i|z|+\tfrac{1}{k}$ and we get
\begin{align*}
 |F_{0}(z)|
&=|(y\circ h)(\theta(z))|e^{1-\frac{1}{k}|z|}
 \leq e^{1+\frac{1}{k^{2}}}|(y\circ h)(\theta(z))|e^{-\frac{1}{k}|\theta(z)|}\\
&\leq e^{1+\frac{1}{k^{2}}}|y\circ h|_{k,[0,\infty]}=:C_{1}<\infty.
\end{align*}
Due to the Phragm\'en--Lindel\"of theorem \cite[Theorem 3.4, p.\ 124]{stein2003} (applied to 
$F(z)\coloneq\tfrac{1}{\max(C_{0},C_{1})}F_{0}(z)$) we obtain 
\[
|F_{0}(z)|\leq \max(C_{0},C_{1})=:C_{2},\;z\in\overline{S}, 
\]
and hence
\[
 |(y\circ h)(\theta(z))|
\leq C_{2} e^{-k\re(\theta(z))+\frac{1}{k}\im(\theta(z))}
\leq C_{2} e^{-k|\re(\theta(z))|+\frac{1}{k}|\theta(z)|},\;z\in\overline{S},
\]
which implies 
\[
 \sup_{z\in\overline{S}_{0}} |(y\circ h)(z)|e^{-\frac{1}{k}|z|+k|\re(z)|}\leq C_{2}<\infty.
\]
Similarly, we get  
\[
 \sup_{z\in\overline{S}_{1}} |(y\circ h)(z)|e^{-\frac{1}{k}|z|+k|\re(z)|}<\infty
\]
for $S_{1}\coloneq\{z\in\C\;|\;\re(z)>\tfrac{1}{k},\,\im(z)<0\}$ by choosing 
$\theta_{1}\colon \overline{S}\to\overline{S}_{1}$, $\theta_{1}(z)\coloneq e^{-i\tfrac{\pi}{4}}z+\tfrac{1}{k}$, and
\[
F_{1}\colon \overline{S}\to \C,\;F_{1}(z)\coloneq(y\circ h)(\theta_{1}(z))e^{(k-\frac{i}{k})\theta_{1}(z)}.
\] 
We conclude that $y\circ h\in\mathcal{LO}_{\{\infty\}}$. 
The weak-strong principle \cite[3.20 Corollary c), p.\ 14]{kruse2018_3}
yields $h\in\mathcal{LO}_{\{\infty\}}(F_{A})$ since $F_{A}=(F,\tau_{A})$ is sequentially complete and 
$\mathcal{LO}_{\{\infty\}}$ a nuclear Fr\'echet space by \cite[7.3 Proposition, p.\ 106]{kruse2021_1}. 
Hence $[u]=0$ by \prettyref{thm:asymp_laplace}. 
\end{proof}

Now, we generalise Langenbruch's sufficient criterion 
\cite[Theorem 7.2, p.\ 62]{langenbruch2011} for the uniqueness property, which itself is a generalisation of 
Lyubich's uniqueness theorem \cite[Theorem 9.2, p.\ 40]{lyubich1966}. For this purpose 
we adapt the notion of an asymptotic left resolvent given by Langenbruch \cite[p.\ 62]{langenbruch2011} 
(for general notions of resolvents in locally convex spaces see \cite{arikan2003} and \cite{domanskilangenbruch2012}).
Let $A\colon F\coloneq D(A)\subset E\to E$ be a sequentially closed linear operator.
We denote by $E_{\alpha}\coloneq(E/\ker p_{\alpha})^{\;\widehat{}}$ and $F_{\alpha}\coloneq(F/\ker p_{\alpha,A})^{\;\widehat{}}$ 
the canonical local Banach spaces for $p_{\alpha}$ resp.\ $p_{\alpha,A}$ and 
by $\|x+\ker p_{\alpha}\|_{\alpha}\coloneq p_{\alpha}(x)$, $x\in E$, resp.\ 
$\|x+\ker p_{\alpha}\|_{\alpha,A}\coloneq p_{\alpha,A}(x)$, $x\in F$, the norms on $E/\ker p_{\alpha}$ resp.\ $F/\ker p_{\alpha,A}$, which we extend to norms on the local Banach spaces with the same symbol. 
Further, we denote by $\kappa_{\alpha}^{F}\colon F_{A}\to F_{\alpha}$, $x\mapsto x+\ker p_{\alpha,A}$, 
the corresponding spectral map of $F_{A}$ for $\alpha\in\mathfrak{A}$.
A set of operators $(R_{\alpha}(t,A))_{\alpha\in\mathfrak{A}}$ is an 
\emph{asymptotic left resolvent} if for all $\alpha\in\mathfrak{A}$ there is $t_{\alpha}>0$ such that 
$R_{\alpha}(t,A)\in L(E,F_{\alpha})$ for all $t\geq t_{\alpha}$ and the continuous linear map 
$S_{\alpha}(t)\colon F_{A}\to F_{\alpha}$ given by 
\begin{equation}\label{eq:asymp_left_res_1}
S_{\alpha}(t)\coloneq R_{\alpha}(t,A)(t-A)-\kappa_{\alpha}^{F},\quad t\geq t_{\alpha},
\end{equation}
fulfils
\begin{align}\label{eq:asymp_left_res_2}
\forall\;n\in\N\;&\exists\;\beta\in\mathfrak{A},\,t_{\alpha,n}\geq t_{\alpha},\,C_{1},C_{2}>0\; 
\forall\;t\geq t_{\alpha,n},\,x\in F:\notag\\
&\|S_{\alpha}(t)x\|_{\alpha}\leq C_{1} p_{\beta,A}(x)
\quad\text{and}\quad 
\|S_{\alpha}^{\beta}(t)\|_{L(F_{\beta},F_{\alpha})}\leq C_{2}e^{-nt}
\end{align}
where $S_{\alpha}^{\beta}(t)\colon F_{\beta}\to F_{\alpha}$ is the continuous linear extension of the map 
$F/\ker p_{\beta,A}\to F_{\alpha}$, $x+\ker p_{\beta,A}\mapsto S_{\alpha}(t)x$.

\begin{thm}\label{thm:unique_prop_suff_0}
Let $E$ be an admissible sequentially complete $\C$-lcHs and $A\colon F\coloneq D(A)\subset E\to E$ a sequentially closed 
linear operator with admissible $F_{A}=(F,\tau_{A})$. 
The ACP \eqref{eq:ACP} has the uniqueness property (in the sense of hyperfunctions) if there is an asymptotic left 
resolvent $(R_{\alpha}(t,A))_{\alpha\in\mathfrak{A}}$ such that
\begin{align}\label{eq:asymp_left_res_3}
\forall\;\alpha\in\mathfrak{A}\;&\exists\;\gamma\in\mathfrak{A},\,k\in\N,\,C_{3},C_{4}>0\; 
\forall\;t\geq t_{\alpha},\,x\in E:\notag\\
&\|R_{\alpha}(t,A)x\|_{\alpha}\leq C_{3} p_{\gamma}(x)
\quad\text{and}\quad 
\|R_{\alpha}^{\gamma}(t,C)\|_{L(E_{\gamma},F_{\alpha})}\leq C_{4} e^{kt}
\end{align}
where $R_{\alpha}^{\gamma}(t,A)\colon E_{\gamma}\to F_{\alpha}$ is the continuous linear extension of the map
$E/\ker p_{\gamma}\to F_{\alpha}$, $x+\ker p_{\gamma}\mapsto R_{\alpha}(t,A)x$.
\end{thm}
\begin{proof}
Let $h\in\mathcal{LO}_{[0,\infty]}(F_{A})$, $v\coloneq (z-A)h\in\mathcal{LO}_{\{\infty\}}(E)$, $\alpha\in\mathfrak{A}$ 
and $m\in\N$. 
Then there are $\gamma\in\mathfrak{A}$ and $k\in\N$, and for any $n\in\N$, $n>m$, there is $\beta\in\mathfrak{A}$ such that 
\begin{align*}
 \|\kappa_{\alpha}^{F}(h(t))\|_{\alpha,A}
&\underset{\mathclap{\eqref{eq:asymp_left_res_1}}}{\leq}\; \|R_{\alpha}(t,A)v(t)\|_{\alpha,A}
 + \|S_{\alpha}(t)h(t)\|_{\alpha,A}\\
&\leq \|R_{\alpha}^{\gamma}(t,A)\|_{L(E_{\gamma},F_{\alpha})}p_{\gamma}(v(t))
 +\|S_{\alpha}^{\beta}(t)\|_{L(F_{\beta},F_{\alpha})}p_{\beta,A}(h(t))\\
&\underset{\mathclap{\eqref{eq:asymp_left_res_2},\eqref{eq:asymp_left_res_3}}}{\leq} C_{4} e^{kt}p_{\gamma}(v(t))
 + C_{2}e^{-nt} p_{\beta,A}(h(t))
\end{align*}
for all $t\geq t_{\alpha,n}$. It follows that
\begin{align*}
 \|\kappa_{\alpha}^{F}(h(t))\|_{\alpha,A}e^{mt}
&\leq C_{4}p_{\gamma}(v(t)) e^{(k+m)t} + C_{2} p_{\beta,A}(h(t))e^{(m-n)t}\\
&\leq C_{4}p_{\gamma}(v(t))e^{-\frac{1}{k+m+1}t+(k+m+1)t}+ C_{2} p_{\beta,A}(h(t))e^{-\frac{1}{m}t}\\
&\leq C_{4}|v|_{k+m+1,\gamma,\{\infty\}}+C_{2}|h|_{m,(\beta,A),[0,\infty]}
\end{align*}
for all $t\geq t_{\alpha,n}$, which implies for $\varepsilon>0$
\begin{flalign*}
&\hspace{0.35cm} \sup_{t\geq \varepsilon} \|\kappa_{\alpha}^{F}(h(t))\|_{\alpha,A}e^{mt} \\
&\leq e^{mt_{\alpha,n}}\max_{\min(\varepsilon,t_{\alpha,n})\leq t\leq t_{\alpha,n}}
  \|\kappa_{\alpha}^{F}(h(t))\|_{\alpha,A}+ C_{4}|v|_{k+m+1,\gamma,\{\infty\}}+C_{2}|h|_{m,(\beta,A),[0,\infty]}
 <\infty
\end{flalign*}
where we use the continuity of $\|\cdot\|_{\alpha,A}\circ\kappa_{\alpha}^{F}\circ h$ as well. 
Thus $\{h(t)e^{mt}\;|\;t\geq\varepsilon\}$ is bounded in $F_{A}$ and we apply \prettyref{thm:unique_prop_equiv} (c).
\end{proof}

Langenbruch also formulated a sufficient criterion \cite[Theorem 7.3, p.\ 62]{langenbruch2011} for 
the uniqueness property by means of an asymptotic existence assumption for the dual operator, which we improve next.
 
\begin{thm}\label{thm:unique_prop_suff}
Let $E$ be an admissible sequentially complete $\C$-lcHs and $A\colon F\coloneq D(A)\subset E\to E$ a sequentially closed 
linear operator with admissible $F_{A}=(F,\tau_{A})$. 
Then the ACP \eqref{eq:ACP} has the uniqueness property (in the sense of hyperfunctions) if for any $y\in F_{A}'$ and 
any $n\in\N$ there are $k\in\N$, $\alpha\in\mathfrak{A}$, $C_{1}>0$ and $t_{y,n}> 0$ such that for 
any $t\geq t_{y,n}$ there are $\widetilde{y}_{y,n}(t)\in E'$, $s_{y,n}(t)\in F_{A}'$, such that for all $t\geq t_{y,n}$ and 
$x\in E$, $z\in F$ it holds that
\begin{align*}
(t-A^{\operatorname{t}})\widetilde{y}_{y,n}(t)&=y+s_{y,n}(t),\quad 
|\langle \widetilde{y}_{y,n}(t), x\rangle| \leq C_{1}p_{\alpha}(x)e^{kt},\\
&|\langle s_{y,n}(t),z\rangle|\leq C_{1}(p_{\alpha}(z)+p_{\alpha}(Az))e^{-nt}.
\end{align*}
\end{thm}
\begin{proof}
Let $h\in\mathcal{LO}_{[0,\infty]}(F_{A})$ and $v\coloneq (z-A)h\in\mathcal{LO}_{\{\infty\}}(E)$. 
Due to our assumption we have for any $y\in F_{A}'$ and $n\in\N$ 
\[
 \langle y,h(t)\rangle 
=\langle (t-A^{\operatorname{t}})\widetilde{y}_{y,n}(t), h(t)\rangle - \langle s_{y,n}(t),h(t)\rangle 
=\langle \widetilde{y}_{y,n}(t), v(t)\rangle - \langle s_{y,n}(t),h(t)\rangle 
\]
for $t\geq t_{y,n}$, implying
\begin{align*}
 |\langle y,h(t)\rangle|
&\leq C_{1}p_{\alpha}(v(t))e^{kt}+C_{1}(p_{\alpha}(h(t))+p_{\alpha}(Ah(t)))e^{-nt}.
\end{align*}
Let $m\in\N$ and choose $n\in\N$ with $n>m$. Then we get
\begin{align*}
 |\langle y,h(t)\rangle|e^{mt}
&\leq C_{1}p_{\alpha}(v(t))e^{(k+m)t}+C_{1}p_{\alpha,A}(h(t))e^{(m-n)t}\\
&\leq  C_{1}p_{\alpha}(v(t))e^{-\frac{1}{k+m+1}t+(k+m+1)t}+C_{1}p_{\alpha,A}(h(t))e^{-\frac{1}{m}t}\\
&\leq C_{1}|v|_{k+m+1,\alpha,\{\infty\}}+C_{1}|h|_{m,(\alpha,A),[0,\infty]},
\end{align*}
for $t\geq t_{y,n}$, which yields for $\varepsilon>0$ 
\begin{flalign*}
&\hspace{0.35cm} \sup_{t\geq \varepsilon}|\langle y,h(t)\rangle|e^{mt} \\
&\leq e^{mt_{y,n}}\max_{\min(\varepsilon,t_{y,n})\leq t\leq t_{y,n}}|\langle y,h(t)\rangle|
 + C_{1}|v|_{k+m+1,\alpha,\{\infty\}}+C_{1}|h|_{m,(\alpha,A),[0,\infty]}<\infty
\end{flalign*}
where we use the continuity of $y\circ h$ as well. Therefore $\{h(t)e^{mt}\;|\;t\geq\varepsilon\}$ 
is weakly bounded in $F_{A}$ for any $m\in\N$ and we apply \prettyref{thm:unique_prop_equiv} (c).
\end{proof}

As an application of \prettyref{thm:unique_prop_suff} we consider the uniqueness of the ACP in the setting  
where $E\coloneq F\coloneq s(\N)_{b}'$ with the nuclear Fr\'echet space
\[
s(\N)\coloneq\{x\in\C^{\N}\;|\;\forall\;p\in\N:\;|x|_{p}^{s(\N)}\coloneq\sup_{i\in\N}|x_{i}|i^{p}<\infty\}
\]
of rapidly decreasing sequences and $A\colon F\to E$ is a continuous linear operator. 
Since $s(\N)$ is reflexive, we have $(s(\N)_{b}')_{b}'=s(\N)$ and $A^{\operatorname{t}}\in L(s(\N))$ 
for the dual map by \cite[Proposition 23.30 (b), p.\ 274]{meisevogt1997}. 
Due to \cite[Exercises 4, p.\ 377]{meisevogt1997} the map $A^{\operatorname{t}}$ is given by an infinite matrix 
$\mathbf{A}^{\operatorname{t}}=(a_{ij})_{i,j\in\N}\in\C^{\N\times\N}$ such that 
\[
\forall\;\sigma\in\R\;\exists\;s\in\R,\,C>0\;\forall\;i,j\in\N:\;|a_{ij}|\leq Cj^{s}i^{-\sigma}
\]
because $s(\N)$ coincides with the power series space $\Lambda_{\infty}((\ln(j))_{j\in\N})$. 
We also consider the ACP in the classical sense in our next theorem, i.e.\ the problem 
\begin{equation}\label{eq:ACP_classical}
x'(t)=Ax(t),\quad t>0, \quad x(0)=x_{0}\in s(\N)',
\end{equation}
where $x\in\mathcal{C}^{1}([0,\infty[,s(\N)_{b}')$.

\begin{thm}\label{thm:unique_exa}
Let $A\in L(s(\N)_{b}')$ and $\mathbf{A}^{\operatorname{t}}$ the infinite matrix that 
represents $A^{\operatorname{t}}\in L(s(\N))$. 
Let $(\mathbf{A}^{\operatorname{t}})^{l}=(a_{ij}^{(l)})_{i,j\in\N}\in\C^{\N\times\N}$ for all $l\in\N$. 
Consider the following statements:
\begin{enumerate}
\item[(a)] $\forall\;p\in\N\;\exists\;q\in\N,\,C>0\;\forall\;l,i,j\in\N:\;|a_{ij}^{(l)}|\leq C^{l}j^{q}i^{-p}$
\item[(b)] The ACP \eqref{eq:ACP} has the uniqueness property (in the sense of hyperfunctions).
\item[(c)] The ACP \eqref{eq:ACP_classical} has the uniqueness property (in the classical sense).
\end{enumerate}
We have the chain of implications (a) $\Rightarrow$ (b) $\Rightarrow$ (c).
\end{thm}
\begin{proof}
(a) $\Rightarrow$ (b): We will use \prettyref{thm:unique_prop_suff}. The complete space $E\coloneq F\coloneq s(\N)_{b}'$ is 
admissible by \prettyref{thm:examples_strictly_admiss} (c) and \cite[Corollary 4.8, p.\ 1116]{D/L}. 
Since $A\in L(s(\N)_{b}')$ and $s(\N)$ is reflexive, we have $(s(\N)_{b}')_{A}=s(\N)_{b}'$ 
and $s(\N)=(s(\N)_{b}')_{b}'$. Let $y=(y_{j})_{j\in\N}\in (s(\N)_{b}')_{b}'=s(\N)$ and set 
\[
  Y(t)^{(m)}
\coloneq \sum_{l=0}^{m}(\mathbf{A}^{\operatorname{t}})^{l}yt^{-l-1}
 =yt^{-1}+\sum_{l=1}^{m}\bigl(\sum_{j=1}^{\infty}a_{ij}^{(l)}y_{j}\bigr)_{i\in\N}t^{-l-1}\in s(\N)
\]
for $m\in\N_{0}$ and $t>0$. We claim that $(Y(t)^{(m)})_{m\in\N_{0}}$ converges in $s(\N)$ if $t$ is big enough. 
We note that for any $m,n\in\N$, $m\geq n$, and $i,k\in\N$ it holds
\begin{align}\label{eq:unique_exa_estim}
  |\sum_{l=n}^{m}\sum_{j=1}^{\infty}a_{ij}^{(l)}y_{j}t^{-l-1}|i^{p}
&\leq \sum_{l=n}^{m}\sum_{j=1}^{\infty}|a_{ij}^{(l)}||y_{j}|t^{-l-1}i^{p}
 \leq \sum_{l=n}^{m}\sum_{j=1}^{\infty}C^{l}j^{q}i^{-p}|y_{j}|t^{-l-1}i^{p}\notag\\
&=\sum_{j=1}^{\infty}j^{q}|y_{j}|\sum_{l=n}^{m}C^{l}t^{-l-1}
 \leq\sum_{j=1}^{\infty}j^{-2}\sup_{r\in\N}|y_{r}|r^{q+2}
  t^{-1}\sum_{l=n}^{m}\Bigl(\frac{C}{t}\Bigr)^{l}\notag\\
&=\frac{\pi^{2}}{6}\|y\|_{q+2}^{s(\N)}t^{-1}\sum_{l=n}^{m}\Bigl(\frac{C}{t}\Bigr)^{l},
\end{align}
which implies that $(Y(t)^{(m)})_{m\in\N_{0}}$ is a Cauchy sequence in $s(\N)$ if $t>C$. Hence the limit 
\[
Y(t)\coloneq\lim_{m\to\infty}Y(t)^{(m)}=\sum_{l=0}^{\infty}(\mathbf{A}^{\operatorname{t}})^{l}yt^{-l-1}
     =\sum_{l=0}^{\infty}\bigl(\sum_{j=1}^{\infty}a_{ij}^{(l)}y_{j}\bigr)_{i\in\N}t^{-l-1}
\]
exists in the complete space $s(\N)$ if $t>C$. Furthermore, we have 
\begin{align*}
 (t-A^{\operatorname{t}})Y(t)^{(m)}
&=tY(t)^{(m)}-A^{\operatorname{t}}Y(t)^{(m)}
 =\sum_{l=0}^{m}(\mathbf{A}^{\operatorname{t}})^{l}yt^{-l}
  -\sum_{l=0}^{m}(\mathbf{A}^{\operatorname{t}})^{l+1}yt^{-l-1}\\
&=y-(\mathbf{A}^{\operatorname{t}})^{m+1}yt^{-m-1}
\end{align*}
as well as 
\begin{align*}
 |((\mathbf{A}^{\operatorname{t}})^{m+1}y)_{i}|t^{-m-1}i^{p}
&\leq\sum_{j=1}^{\infty}|a_{ij}^{(m+1)}||y_{j}|t^{-m-1}i^{p}
 \leq \sum_{j=1}^{\infty}C^{m+1}j^{q}i^{-p}|y_{j}|t^{-m-1}i^{p}\\
&\leq \frac{\pi^{2}}{6}|y|_{q+2}^{s(\N)}\Bigl(\frac{C}{t}\Bigr)^{m+1}
\end{align*}
for all $m\in\N_{0}$ and $t>0$, yielding that  
\[
 (t-A^{\operatorname{t}})Y(t)
=\lim_{m\to\infty}(t-A^{\operatorname{t}})Y(t)^{(m)}
=y-\lim_{m\to\infty}(\mathbf{A}^{\operatorname{t}})^{m+1}yt^{-m-1}
=y
\]
in $s(\N)$ if $t>C$. The topology of $s(\N)_{b}'$ is induced by the seminorms
\[
p_{B}(x)\coloneq\sup_{w\in B}|x(w)|,\quad x\in s(\N)',
\]
for bounded sets $B\subset s(\N)$. We remark that 
\begin{align*}
  |Y(t)|_{p}^{s(\N)}
&=\sup_{i\in\N}|Y(t)_{i}|i^{p}
 \underset{\eqref{eq:unique_exa_estim}}{\leq}
\sup_{i\in\N}t^{-1}|y_{i}|i^{p}+\frac{\pi^{2}}{6}
   |y|_{q+2}^{s(\N)}t^{-1}\sum_{l=1}^{\infty}\Bigl(\frac{C}{t}\Bigr)^{l}\\
&= t^{-1}|y|_{p}^{s(\N)}+\frac{\pi^{2}}{6}|y|_{q+2}^{s(\N)}\frac{\frac{C}{t}}{t(1-\frac{C}{t})}
\leq  \frac{1}{2C}|y|_{p}^{s(\N)}+\frac{\pi^{2}}{12C}|y|_{q+2}^{s(\N)}
=: K_{p}
\end{align*}
if $t>2C$. Thus $Y(t)\in\{w\in s(\N)\;|\;\forall\;p\in\N:\;|w|_{p}^{s(\N)}\leq K_{p} \}=:B_{0}$ if $t>2C$, 
and $B_{0}$ is a bounded set in $s(\N)$. 
So for $x\in s(\N)'$ we have 
\[
 |\langle Y(t), x\rangle|
=|x(Y(t))|
\leq \sup_{w\in B_{0}}|x(w)|=p_{B_{0}}(x)
\]
if $t>2C$. Hence we may apply \prettyref{thm:unique_prop_suff} with $t_{y,n}\coloneq 2C$,  
$\widetilde{y}_{y,n}\coloneq Y$, $s_{y,n}\coloneq 0$ for $n\in\N$ as well as $k\coloneq C_{1}\coloneq 1$ and $\alpha\coloneq B_{0}$.

(b) $\Rightarrow$ (c): Let $x\in\mathcal{C}^{1}([0,\infty[,s(\N)_{b}')$ 
be a solution of the ACP \eqref{eq:ACP_classical} for $x_{0}\coloneq 0$. 
Then $x$ defines a hyperfunction $[u]$ in $\mathcal{B}([0,\infty[,s(\N)_{b}')$ 
(for instance by \cite[Theorem 6.9, p.\ 1125]{D/L} as in \cite[Theorem 1.3.10, p.\ 25]{Kan} and 
\cite[Theorem 1.3.13 b), p.\ 31]{Kan}) which solves \eqref{eq:ACP_hyp} for $x_{0}=0$. Thus $[u]=0$  
by the uniqueness property in the sense of hyperfunctions, implying $x=0$ on $[0,\infty[$.
\end{proof}

\section{Solvability of the ACP}
\label{sect:solvability}

Let us turn to the question of existence of a solution of the ACP \eqref{eq:ACP}. 
Following \cite[p.\ 64]{langenbruch2011}, this boils down to 
solving the equation $(\lambda-A)S(\lambda)=x_{0}$ only approximately near the half-circle 
$S_{\infty}\coloneq\{\infty e^{i\varphi}\;|\;|\varphi|<\tfrac{\pi}{2}\}$ at $\infty$, 
and the approximate solution is needed only in the local Banach spaces of $F_{A}=(F,\tau_{A})$.
The precise characterisation of existence of a solution given in \prettyref{thm:solution_ACP} below uses 
the Laplace transform of $E$-valued Laplace hyperfunctions from \cite{domanskilangenbruch2010}. 
We recall what is needed.
Let
\[
H\coloneq\lim\limits_{\substack{\longrightarrow\\K\in \N}}\bigl(\lim\limits_{\substack{\longleftarrow\\k\in \N}}\,H_{K,k}\bigr)
\]
be the inductive limit of the projective limit $\lim\limits_{\substack{\longleftarrow\\k\in \N}}\,H_{K,k}$ where
\[
H_{K,k}\coloneq\{f\in\mathcal{O}(\Omega_{K})\;|\;\|f\|_{K,k}^{H}\coloneq\sup_{z\in\Omega_{K}}|f(z)|e^{k\re(z)}<\infty\}
\]
and
\[
\Omega_{K}\coloneq\Bigl\{z\in\C\;|\;|\im(z)|<\frac{\re(z)}{K}+\frac{1}{K^{2}}\Bigr\}.
\]

\begin{center}
\begin{minipage}{\linewidth}
\centering
\begin{tikzpicture}
\def\mypath{
(6,2.11) -- (-0.33,0) -- (6,-2.11)}
\fill[fill=black!10,draw=black,dashed,thick] \mypath;
\node[anchor=north east] (A) at (5.5,1) {$\Omega_{K}$};
\draw(-0.33cm,1pt)--(-0.33cm,-1pt)node[anchor=north east] {$-\tfrac{1}{K}$};
\draw(1.1cm,0.44cm)--(0.9cm,0.44cm)node[anchor=south east] {$\tfrac{1}{K^{2}}$};
\draw(1.1cm,-0.44cm)--(0.9cm,-0.44cm)node[anchor=north east] {$-\tfrac{1}{K^{2}}$};
\draw[->] (-1,0) -- (6,0) node[right] {$\re(z)$} coordinate (x axis);
\draw[->] (1,-2.11) -- (1,2.11) node[above] {$\im(z)$} coordinate (y axis);
\end{tikzpicture}
\end{minipage}
\captionsetup{type=figure}
\caption{$\Omega_{K}$ for $K\in\N$}
\end{center}

By \cite[Definition 2.3, p.\ 133]{domanskilangenbruch2010} an $E$-valued \emph{Laplace hyperfunction} 
(in the sense of Doma{\'n}ski and Langenbruch) is a continuous linear operator $T\colon H\to E$ for complete $E$. 
Its Laplace transform $\mathscr{L}(T)$ is not a single holomorphic function but a compatible family 
of holomorphic functions, a so-called spectral-valued holomorphic function, 
for whose definition we need to direct the index set $\mathfrak{A}$ of the seminorms of $E$ first.

Let $E$ be a complete $\C$-lcHs with a directed system of seminorms $(p_{\alpha})_{\alpha\in\mathfrak{A}}$, i.e.\ 
for $\alpha,\beta\in\mathfrak{A}$ there are $\gamma\in\mathfrak{A}$ and $C_{1}>0$ such that 
$\max(p_{\alpha},p_{\beta})\leq C_{1}p_{\gamma}$. We write $\alpha\leq\beta$ for $\alpha,\beta\in\mathfrak{A}$ 
if there is $C_{2}>0$ such $p_{\alpha}\leq C_{2}p_{\beta}$. Then $\leq$ is a preorder on $\mathfrak{A}$ 
and $(\mathfrak{A},\leq)$ a directed set due to the system of seminorms being directed. 
Furthermore, for $\alpha,\beta\in\mathfrak{A}$ with $\alpha\leq\beta$ 
we denote by $\kappa^{\beta}_{\alpha}\colon E_{\beta}\to E_{\alpha}$ the linking maps of the local Banach spaces,
which are the continuous linear extensions of the maps $E/\ker p_{\alpha} \to E/\ker p_{\beta}$, 
$x+\ker p_{\alpha}\mapsto x + \ker p_{\beta}$, and by 
$\kappa_{\alpha}^{E}\colon E\to E_{\alpha}$, $x\mapsto x+\ker p_{\alpha}$, the spectral maps. 
With these definitions $E$ becomes a projective limit of its local Banach spaces $E_{\alpha}$, i.e.\ 
$E=\lim\limits_{\longleftarrow}{}_{\alpha\in\mathfrak{A}}E_{\alpha}$ (see \cite[p.\ 151--152]{Kaballo}).  

Let $\mathcal{E}\coloneq (E_{\alpha})_{\alpha\in\mathfrak{A}}$, and $\mathcal{G}\coloneq (G_{\alpha})_{\alpha\in\mathfrak{A}}$ 
be a directed family of non-empty domains in $\C$, i.e.\ 
they are open and connected sets and $G_{\beta}\subset G_{\alpha}$ for $\alpha\leq\beta$ 
(see \cite[p.\ 131]{domanskilangenbruch2010}).
By \cite[Definition 2.1, p.\ 132]{domanskilangenbruch2010} a family 
$\mathcal{S}\coloneq (S_{\alpha})_{\alpha\in\mathfrak{A}}$ is called a \emph{spectral-valued} 
(or $\mathcal{E}$\emph{-valued}) \emph{holomorphic function} (denoted by $\mathcal{S}\colon\mathcal{G}\to\mathcal{E}$) 
if 
\begin{enumerate}
\item[(i)] $S_{\alpha}\colon G_{\alpha}\to E_{\alpha}$ is holomorphic for all $\alpha\in\mathfrak{A}$, and
\item[(ii)] (compatibility) 
$\forall\;\alpha,\beta\in\mathfrak{A},\,\alpha\leq\beta:\;
\kappa_{\alpha}^{\beta}\circ S_{\beta}={S_{\alpha}}_{\mid G_{\beta}}$.
\end{enumerate}
For $0<\varphi<\tfrac{\pi}{2}$ and $r\geq 0$ we set 
\[
\Gamma_{r,\varphi}\coloneq\{\rho e^{i\psi}\;|\; \rho\geq r,\,|\psi|\leq \varphi\}.
\]
An open set $U\subset\C$ is called \emph{postsectorial} (see \cite[p.\ 37]{lumer1999}, \cite[p.\ 150]{lumer2001}) if 
\[
\forall\;0<\varphi<\frac{\pi}{2}\;\exists\;r>0:\;\Gamma_{r,\varphi}\subset U.
\]
\begin{center}
\begin{minipage}{\linewidth}
\centering
\begin{tikzpicture}
\def\mypath{
 (6,2) -- (2,0.66) arc(18.43:-18.43:2.106) -- (6,-2)}
\fill[fill=black!10,draw=black,thick] \mypath;
\draw[dashed,thick]  (0,0) -- (2,0.66);
\draw[dashed,thick]  (0,0) -- (2,-0.66);
\node[anchor=north east] (A) at (5.5,1) {$\Gamma_{r,\varphi}$};
\node[anchor=north east] (B) at (1.7,0.45) {$\varphi$};
\node[anchor=north east] (C) at (1,-0.425) {$r$};
\draw[->] (-2,0) -- (6,0) node[right] {$\re(\lambda)$} coordinate (x axis);
\draw[->] (0,-2) -- (0,2) node[above] {$\im(\lambda)$} coordinate (y axis);
\end{tikzpicture}
\end{minipage}
\captionsetup{type=figure}
\caption{$\Gamma_{r,\varphi}$ for $0<\varphi<\tfrac{\pi}{2}$ and $r\geq 0$ (cf.\ \cite[Figure 1.11, p.\ 111]{kruse2021_1})}
\end{center}
Further, we define the set $H_{\operatorname{exp}}(\mathcal{E})$ of all $\mathcal{E}$-valued holomorphic functions 
$\mathcal{S}\colon\mathcal{G}\to\mathcal{E}$ where $\mathcal{G}$ consists of postsectorial domains and 
\[
\forall\;\alpha\in\mathfrak{A},\,K\in\N,\,0<\varphi<\tfrac{\pi}{2}\;\exists\;r>0:\; 
\Gamma_{r,\varphi}\subset G_{\alpha}
\;\;\text{and}\;\;
\sup_{\lambda\in\Gamma_{r,\varphi}}\|S_{\alpha}(\lambda)\|_{\alpha}e^{-\frac{1}{K}\re(\lambda)}<\infty
\]
(see \cite[Definition 2.6, p.\ 134]{domanskilangenbruch2010}). 
Considering the elements of $H_{\operatorname{exp}}(\mathcal{E})$ as germs near $S_{\infty}$, we note 
that $H_{\operatorname{exp}}(\mathcal{E})$ is a vector space canonically. 

\begin{thm}[{\cite[Theorem 2.4, p.\ 134, Corollary 3.5, p.\ 145]{domanskilangenbruch2010}}]
\label{thm:laplace_on_operators}
Let $E$ be a complete $\C$-lcHs which is the projective limit of a spectrum of 
Banach spaces $\mathcal{E}\coloneq (E_{\alpha})_{\alpha\in\mathfrak{A}}$.
Then the Laplace transform $\mathscr{L}\colon L(H,E)\to H_{\operatorname{exp}}(\mathcal{E})$ is a linear bijection 
such that $\mathscr{L}(\tfrac{d}{d t} T)=\lambda \mathscr{L}(T)$. 
\end{thm}

\begin{rem}
The definition of $H_{\operatorname{exp}}(\mathcal{E})$ in \cite[Definition 2.6, p.\ 134]{domanskilangenbruch2010} 
is actually phrased with a family $\mathcal{G}$ of conoidal sets. 
An open set $G\subset\C$ is called \emph{conoidal} if for every $K\in\N$ there is $k\in\N$ such that 
\[
V_{K,k}\coloneq\Bigl\{\lambda\in\C\;|\;\re(\lambda)>k+\frac{|\im(\lambda)|}{K}\Bigr\}\subset G
\]
(see \cite[Definition 2.7, p.\ 134]{domanskilangenbruch2010}). We note that an open set $G\subset\C$ is conoidal if 
and only if $G$ is postsectorial. 
\begin{center}
\begin{minipage}{\linewidth}
\centering
\begin{tikzpicture}
\def\mypath{
(6,2) -- (2,0) -- (6,-2)}
\fill[fill=black!10,draw=black,dashed,thick] \mypath;
\draw[black,dotted] (0,-0.5) -- (3,-0.5) -- (3,0.5) -- (0,0.5);
\node[anchor=north east] (A) at (5.5,1) {$V_{K,k}$};
\draw(1pt,-0.5cm)--(-1pt,-0.5cm)node[anchor=north east] {$-K$};
\draw(1pt,0.5cm)--(-1pt,0.5cm)node[anchor=south east] {$K$};
\draw(2cm,1pt)--(2cm,-1pt)node[anchor=north east] {$k$};
\draw(3cm,1pt)--(3cm,-1pt)node[anchor=north west] {$k+1$};
\draw[->] (-1,0) -- (6,0) node[right] {$\re(\lambda)$} coordinate (x axis);
\draw[->] (0,-2) -- (0,2) node[above] {$\im(\lambda)$} coordinate (y axis);
\end{tikzpicture}
\end{minipage}
\captionsetup{type=figure}
\caption{$V_{K,k}$ for $K,k\in\N$}
\end{center}
\end{rem}
\begin{proof}
First, we observe that $V_{K,k}=k+\Gamma_{0,\varphi(K)}$ with $\varphi(K)\coloneq\operatorname{arctan}(K)$ 
for every $K,k\in\N$. Let $G$ be conoidal and $0<\varphi<\tfrac{\pi}{2}$. We choose $K\in\N$ such that 
$\varphi(K)>\varphi$. Then there are $k\in\N$ and $r>0$ with 
$\Gamma_{r,\varphi}\subset k+\Gamma_{0,\varphi(K)}=V_{K,k}\subset G$ because $G$ is conoidal. 

Let $G$ be postsectorial and $K\in\N$. Then $0<\varphi(K)<\tfrac{\pi}{2}$ and there is $r>0$ such that 
$\Gamma_{r,\varphi(K)}\subset G$. We choose $k\in\N$ with $k>r$ and get 
$V_{K,k}=k+\Gamma_{0,\varphi(K)}\subset\Gamma_{r,\varphi(K)}\subset G$.
\end{proof}

Let $E$ be as above and $A\colon F\coloneq D(A)\subset E\to E$ a closed linear operator. 
We equip $F$ with the graph topology $\tau_{A}$, which makes it a complete space. We  
denote by $\kappa_{\alpha}^{\beta}\colon F_{\beta}\to F_{\alpha}$ for $\alpha,\beta\in\mathfrak{A}$ with 
$\alpha\leq\beta$ the linking maps of its local Banach spaces. 
Then $F_{A}=(F,\tau_{A})$ is a projective limit of its local Banach spaces $F_{\alpha}$. By the definition of the 
graph topology the map $A\colon F_{A}\to E$ is continuous and for any $\alpha\in\mathfrak{A}$ 
there are $\beta\in\mathfrak{A}$ and $C_{1}>0$ 
such that $p_{\alpha}(Ax)\leq C_{1}p_{\beta,A}(x)$ for all $x\in F$ 
(e.g.\ any $\beta\in\mathfrak{A}$ with $\alpha\leq\beta$). 
This defines a continuous linear operator 
$A_{\alpha}^{\beta}\colon F_{\beta}\to E_{\alpha}$ as the extension of the continuous linear map 
$F/\ker p_{\beta,A}\to E/\ker p_{\alpha}$, $x+\ker p_{\beta,A}\mapsto Ax+\ker p_{\alpha}$ 
(well-defined because $\ker p_{\beta,A}\subset\ker p_{\alpha}\circ A$). 
Moreover, we call $I_{\alpha}^{\beta}\colon F_{\beta}\to E_{\alpha}$ the continuous linear extension of the map 
$F/\ker p_{\beta,A} \to E/\ker p_{\alpha}$, $x+\ker p_{\beta,A}\mapsto x+\ker p_{\alpha}$, 
for $\alpha\leq\beta$ (well-defined because $\alpha\leq\beta$ implies $\ker p_{\beta,A}\subset\ker p_{\alpha}$). 

\begin{thm}\label{thm:solution_ACP}
Let $E$ be an admissible complete $\C$-lcHs with local Banach spaces 
$\mathcal{E}\coloneq (E_{\alpha})_{\alpha\in\mathfrak{A}}$,
let $A\colon F\coloneq D(A)\subset E\to E$ be a closed linear operator and $F_{A}=(F,\tau_{A})$ admissible 
with local Banach spaces $\mathcal{F}\coloneq (F_{\alpha})_{\alpha\in\mathfrak{A}}$. 
For $x_{0}\in E$ the following are equivalent:
\begin{enumerate}
\item[(a)] The ACP \eqref{eq:ACP} has a solution (in the sense of hyperfunctions).
\item[(b)] There is a spectral-valued holomorphic function 
$\mathcal{S}\coloneq (S_{\alpha})_{\alpha\in\mathfrak{A}}\in H_{\operatorname{exp}}(\mathcal{F})$ such that for any 
$\alpha\in\mathfrak{A}$ there is $\beta\in\mathfrak{A}$, $\alpha\leq\beta$, such that 
\begin{equation}\label{eq:solution_ACP_1}
s_{\alpha}^{\beta}\colon G_{\beta}\to E_{\alpha},\; 
s_{\alpha}^{\beta}(\lambda)\coloneq (\lambda I_{\alpha}^{\beta}-A_{\alpha}^{\beta})S_{\beta}(\lambda)
 -\kappa_{\alpha}^{E}(x_{0}),
\end{equation}
is well-defined and 
\begin{equation}\label{eq:solution_ACP_2}
\forall\;j\in\N,\,0<\varphi<\tfrac{\pi}{2}\;\exists\;r>0:\; 
\Gamma_{r,\varphi}\subset G_{\beta}
\;\;\text{and}\;\;
\sup_{\lambda\in\Gamma_{r,\varphi}}\|s_{\alpha}^{\beta}(\lambda)\|_{\alpha}
e^{j\re(\lambda)-\frac{1}{j}|\im(\lambda)|}<\infty.
\end{equation}
\end{enumerate}
\end{thm}
\begin{proof}
(a) $\Rightarrow$ (b): Let $[u]\in\mathcal{B}([0,\infty[,F_{A})$ be a solution of \eqref{eq:ACP_hyp} and 
$[h]\coloneq\mathcal{L}^{\mathcal{B}}([u])\in\mathcal{LO}_{[0,\infty]}(F_{A})/\mathcal{LO}_{\{\infty\}}(F_{A})$. 
It follows that
\[
(\lambda-A)[h]=\mathcal{L}^{\mathcal{B}}(x_{0}\otimes \delta_{0})=[x_0]
\]
in $\mathcal{LO}_{[0,\infty]}(E)/\mathcal{LO}_{\{\infty\}}(E)$. 
We set $G_{\alpha}\coloneq\C_{\re >0}$ and $S_{\alpha}\coloneq\kappa_{\alpha}^{F}\circ h$ for $\alpha\in\mathfrak{A}$. 
Then there is $f\in\mathcal{LO}_{\{\infty\}}(E)$ such that with $\beta=\alpha$
\begin{align*}
\kappa_{\alpha}^{E}(x_{0})+s_{\alpha}^{\alpha}(\lambda)
&=(\lambda I_{\alpha}^{\alpha}-A_{\alpha}^{\alpha})S_{\alpha}(\lambda)
 =(\lambda I_{\alpha}^{\alpha}-A_{\alpha}^{\alpha})\kappa_{\alpha}^{F}(h(\lambda))\\
&= (\lambda I_{\alpha}^{\alpha}-A_{\alpha}^{\alpha})(h(\lambda)+\ker p_{\alpha,A})
 = \lambda h(\lambda)-Ah(\lambda)+\ker p_{\alpha}\\
&=(\lambda-A)h(\lambda)+\ker p_{\alpha}
 =x_{0}+f(\lambda)+\ker p_{\alpha}
 =\kappa_{\alpha}^{E}(x_{0})+f(\lambda)
\end{align*}
and thus $s_{\alpha}^{\alpha}(\lambda)=f(\lambda)$ for $\lambda\in G_{\alpha}$. 
Let $j\in\N$ and $0<\varphi<\tfrac{\pi}{2}$. 
We note that $\Gamma_{r,\varphi}\subset G_{\alpha}$ for any $r>0$ and with $k\in\N$ such that $\tfrac{1}{k}\leq r$ and 
$k\geq j$ we obtain
\begin{align*}
\sup_{\lambda\in\Gamma_{r,\varphi}}\|s_{\alpha}^{\alpha}(\lambda)\|_{\alpha}
 e^{j\re(\lambda)-\frac{1}{j}|\im(\lambda)|}
&\leq \sup_{\re(\lambda)\geq\frac{1}{k}}p_{\alpha}(f(\lambda))
 e^{j|\re(\lambda)|-\frac{1}{j}|\lambda|+\frac{1}{j}|\re(\lambda)|}\\
&\leq \sup_{\re(\lambda)\geq\frac{1}{k}}p_{\alpha}(f(\lambda))
 e^{-\frac{1}{j+1}|\lambda|+(j+1)|\re(\lambda)|}
 \leq |f|_{k+1,\alpha,\{\infty\}}.
\end{align*}
(b) $\Rightarrow$ (a): First, we observe that for $\alpha,\beta\in\mathfrak{A}$ with $\alpha\leq\beta$ the map 
$s_{\alpha}^{\beta}\colon G_{\beta}\to E_{\alpha}$ is well-defined by our considerations above this theorem. 
In addition, $s_{\alpha}^{\beta}$ is holomorphic on $G_{\beta}$ because $I_{\alpha}^{\beta}$ and $A_{\alpha}^{\beta}$ 
are linear and continuous and $S_{\beta}$ holomorphic by (i). We observe that for $\lambda\in G_{\beta}$ 
there is $(h_{n}(\lambda))_{n\in\N}$ in $F$ such that 
$S_{\beta}(\lambda)=\lim_{n\to\infty}(h_{n}(\lambda)+\ker p_{\beta,A})$ in $F_{\beta}$ and 
\begin{equation}\label{eq:solution_ACP_3}
 s_{\alpha}^{\beta}(\lambda)+\kappa_{\alpha}^{E}(x_{0})
=(\lambda I_{\alpha}^{\beta}-A_{\alpha}^{\beta})S_{\beta}(\lambda)
=\lim_{n\to\infty}((\lambda-A) h_{n}(\lambda)+\ker p_{\alpha}).
\end{equation}

Now, we want to construct an $\mathcal{E}$-valued holomorphic function 
$\widetilde{s}$ on a suitable family $\widetilde{G}$ of postsectorial domains using our maps 
$s_{\alpha}^{\beta}$. For $\alpha\in\mathfrak{A}$ we set 
\[
M_{\alpha}\coloneq\{\beta\in\mathfrak{A}\;|\;\alpha\leq\beta\;\text{and}\;\eqref{eq:solution_ACP_2} 
 \;\text{is satisfied} \}
 \quad\text{and}\quad
 \widetilde{G}_{\alpha}\coloneq\bigcup_{\beta\in M_{\alpha}}G_{\beta}.
\]
The sets $\widetilde{G}_{\alpha}\subset\C$ are non-empty by assumption as well as open, connected and postsectorial 
as they are unions of such sets. 
Next, we show that $M_{\alpha}\subset M_{\gamma}$ for $\alpha,\gamma\in\mathfrak{A}$ with $\gamma\leq\alpha$, 
which then implies $\widetilde{G}_{\alpha}\subset \widetilde{G}_{\gamma}$ and 
means that $\widetilde{\mathcal{G}}\coloneq (\widetilde{G}_{\alpha})_{\alpha\in\mathfrak{A}}$ is directed. 
Let $\beta\in M_{\alpha}$. Then $\gamma\leq\beta$ and it holds by \eqref{eq:solution_ACP_3} that
\begin{align*}
 \kappa_{\gamma}^{\alpha}(s_{\alpha}^{\beta}(\lambda))+\kappa_{\gamma}^{E}(x_{0})
&=\kappa_{\gamma}^{\alpha}(s_{\alpha}^{\beta}(\lambda)+\kappa_{\alpha}^{E}(x_{0}))
 =\lim_{n\to\infty}\kappa_{\gamma}^{\alpha}((\lambda-A) h_{n}(\lambda)+\ker p_{\alpha} )\\
&=\lim_{n\to\infty}((\lambda-A) h_{n}(\lambda)+\ker p_{\gamma})
 =s_{\gamma}^{\beta}(\lambda)+\kappa_{\gamma}^{E}(x_{0})
\end{align*}
and thus $\kappa_{\gamma}^{\alpha}(s_{\alpha}^{\beta}(\lambda))=s_{\gamma}^{\beta}(\lambda)$ 
for $\lambda\in G_{\beta}$. We deduce that there is $C_{1}>0$ such that 
$\|s_{\gamma}^{\beta}(\lambda)\|_{\gamma}\leq C_{1} \|s_{\alpha}^{\beta}(\lambda)\|_{\alpha}$ 
for any $\lambda\in G_{\beta}$ from the continuity of $\kappa_{\gamma}^{\alpha}\colon E_{\alpha}\to E_{\gamma}$.
Therefore $s_{\gamma}^{\beta}$ satisfies the estimate \eqref{eq:solution_ACP_2} with $\alpha$ replaced by $\gamma$, 
which means that $\beta\in M_{\gamma}$. 

Now, let $\beta_{1},\beta_{2}\in M_{\alpha}$. Then $\emptyset\neq(G_{\beta_{1}}\cap G_{\beta_{2}})\subset G_{\alpha}$ 
as $\alpha\leq\beta_{1},\beta_{2}$ and $\mathcal{G}$ is directed and consists of postsectorial sets. 
For $\lambda\in G_{\beta_{1}}\cap G_{\beta_{2}}$ there are $(h_{i,n}(\lambda))_{n\in\N}$ in $F$ such that 
$S_{\beta_{i}}(\lambda)=\lim_{n\to\infty}(h_{i,n}(\lambda)+\ker p_{\beta_{i},C})$ in $F_{\beta_{i}}$ and 
\[
 s_{\alpha}^{\beta_{i}}(\lambda)+\kappa_{\alpha}^{E}(x_{0})
=\lim_{n\to\infty}((\lambda-A) h_{i,n}(\lambda)+\ker p_{\alpha})
\]
by \eqref{eq:solution_ACP_3} for $i=1,2$. Due to the compatibility (ii) for $\mathcal{S}$ we get 
\[
 S_{\alpha}(\lambda)
\underset{(ii)}{=}(\kappa_{\alpha}^{\beta_{i}}\circ S_{\beta_{i}})(\lambda)
 =\lim_{n\to\infty}(h_{i,n}(\lambda)+\ker p_{\alpha,A})
\]
for $i=1,2$, which yields $\lim_{n\to\infty}(h_{1,n}(\lambda)-h_{2,n}(\lambda)+\ker p_{\alpha,A})=0$ in $F_{\alpha}$ 
and thus in $(F/\ker p_{\alpha,A})$ as well. It follows that $\lim_{n\to\infty}(h_{1,n}(\lambda)-h_{2,n}(\lambda))\in 
\ker p_{\alpha,A}$ and so
\[
 s_{\alpha}^{\beta_{1}}(\lambda)-s_{\alpha}^{\beta_{2}}(\lambda)
= (\lambda-A)(\lim_{n\to\infty}(h_{1,n}(\lambda)-h_{2,n}(\lambda)))+\ker p_{\alpha}=\ker p_{\alpha},
\]
implying $s_{\alpha}^{\beta_{1}}=s_{\alpha}^{\beta_{2}}$ on $G_{\beta_{1}}\cap G_{\beta_{2}}$. 
Therefore the map $\widetilde{s}_{\alpha}\colon \widetilde{G}_{\alpha}\to E_{\alpha}$ given by 
$\widetilde{s}_{\alpha}\coloneq s_{\alpha}^{\beta}$ on $G_{\beta}$ for $\beta\in M_{\alpha}$ is well-defined 
and holomorphic on $\widetilde{G}_{\alpha}$. 
This gives us (i) for 
$\widetilde{s}\coloneq (\widetilde{s}_{\alpha})_{\alpha\in\mathfrak{A}}\colon \widetilde{\mathcal{G}}\to\mathcal{E}$.

Let us turn to the compatibility condition (ii) for $\widetilde{s}$. 
Let $\alpha,\gamma\in\mathfrak{A}$ with $\alpha\leq\gamma$. Then for any $\beta\in M_{\gamma}\subset M_{\alpha}$ 
and $\lambda\in G_{\beta}$ we have by \eqref{eq:solution_ACP_3}
\begin{align*}
 (\kappa_{\alpha}^{\gamma}\circ\widetilde{s}_{\gamma})(\lambda)
&=(\kappa_{\alpha}^{\gamma}\circ s_{\gamma}^{\beta})(\lambda)
 =\lim_{n\to\infty}\kappa_{\alpha}^{\gamma}((\lambda-A) h_{n}(\lambda)-x_{0}+\ker p_{\gamma})\\
&=\lim_{n\to\infty}((\lambda-A) h_{n}(\lambda)-x_{0}+\ker p_{\alpha})
 =s_{\alpha}^{\beta}(\lambda)
 =\widetilde{s}_{\alpha}(\lambda)
\end{align*}
and we conclude that $\widetilde{s}$ fulfils (ii) and is an $\mathcal{E}$-valued holomorphic function. 

Let $\alpha\in\mathfrak{A}$, $K\in\N$ and $0<\varphi<\tfrac{\pi}{2}$ and choose $\beta\in M_{\alpha}$. 
Due to \eqref{eq:solution_ACP_2} for $j=2$ there is $r>0$ such that 
$\Gamma_{r,\varphi}\subset G_{\beta}\subset \widetilde{G}_{\alpha}$ .
We observe that for $\lambda\in\Gamma_{r,\varphi}$ it holds that $\re(\lambda)>0$ and 
\begin{align*}
 -\frac{1}{K}\re(\lambda)
&\leq \frac{1}{2}|\im(\lambda)|-\frac{1}{2}|\im(\lambda)|
 \leq \frac{\operatorname{arctan}(\varphi)}{2}\re(\lambda)-\frac{1}{2}|\im(\lambda)|\\
&\leq 2\re(\lambda)-\frac{1}{2}|\im(\lambda)|,
\end{align*} 
which implies
\[
 \sup_{\lambda\in\Gamma_{r,\varphi}}\|\widetilde{s}_{\alpha}(\lambda)\|_{\alpha}e^{-\frac{1}{K}\re(\lambda)}
\leq\sup_{\lambda\in\Gamma_{r,\varphi}}\|s_{\alpha}^{\beta}(\lambda)\|_{\alpha}
 e^{2\re(\lambda)-\frac{1}{2}|\im(\lambda)|}<\infty.
\]
We conclude that $\widetilde{s}\in H_{\operatorname{exp}}(\mathcal{E})$. 

By \prettyref{thm:laplace_on_operators} and the definition of $\widetilde{s}$ 
in connection with \eqref{eq:solution_ACP_1}
there are $T\in L(H,F_{A})$ and $\widetilde{T}\in L(H,E)$ such that $\mathscr{L}(T)=\mathcal{S}$ and 
$\mathscr{L}(\widetilde{T})=\widetilde{s}$ as well as 
\begin{equation}\label{eq:solution_ACP_4}
\Bigl(\frac{d}{dt}-A\Bigr)T=x_{0}\otimes\delta_{0}+\widetilde{T}
\end{equation}
where $\delta_{0}$ is the Dirac distribution, i.e.\ $\delta_{0}(f)\coloneq f(0)$, and 
$(x_{0}\otimes\delta_{0})(f)\coloneq x_{0}f(0)$ for $f\in H$. As in \cite[Theorem 7.6, p.\ 65]{langenbruch2011} 
we translate this equation from Laplace hyperfunctions to hyperfunctions using the functions 
$f_{\lambda}(t)\coloneq\tfrac{-1}{2\pi i}\tfrac{e^{(t-\lambda)^{2}}}{t-\lambda}$ for $\lambda\notin [0,\infty[$. 
Since $f_{\lambda}\in H$ (for every $\lambda\notin [0,\infty[$ there is $K\in\N$ such that $\lambda\notin\Omega_{K}$), 
the functions
\[
u_{T}\colon\C\setminus[0,\infty[\to F_{A},\; u_{T}(\lambda)\coloneq\langle T,f_{\lambda}\rangle,  
\]
and analogously $u_{\widetilde{T}}\colon\C\setminus[0,\infty[\to E$ are defined. 
The difference quotients of $f_{\lambda}$ w.r.t.\ $\lambda$ converge in $H$, which yields that $u_{T}$ is 
holomorphic and 
\[
 \frac{d}{d\lambda} u_{T}(\lambda)
=\Bigl\langle T,\frac{d}{d\lambda}f_{\lambda}\Bigr\rangle
=\Bigl\langle T,-\frac{d}{d t}f_{\lambda}\Bigr\rangle
=\Bigl\langle \frac{d}{d t}T, f_{\lambda}\Bigr\rangle,\quad\lambda\in \C\setminus[0,\infty[.
\]
Hence we get for $\lambda\in \C\setminus[0,\infty[$
\begin{align*}
 \Bigl(\frac{d}{d\lambda}-A\Bigr) u_{T}(\lambda)
&=\Bigl\langle \bigl(\frac{d}{d t}-A\bigr)T, f_{\lambda}\Bigr\rangle
 \underset{\eqref{eq:solution_ACP_4}}{=}\langle x_{0}\otimes\delta_{t=0}+\widetilde{T},f_{\lambda}\rangle
 =x_{0}\tfrac{-1}{2\pi i}\tfrac{e^{\lambda^{2}}}{-\lambda}+u_{\widetilde{T}}(\lambda)\\
&=-x_{0}f_{0}(\lambda)+u_{\widetilde{T}}(\lambda)
 =(x_{0}\otimes (-f_{0}))(\lambda)+u_{\widetilde{T}}(\lambda).
\end{align*}
Since $[-f_{0}]=-\delta_{0}$ in $\mathcal{B}([0,\infty[)$ by \cite[4.11 Example, p.\ 96]{kruse2021_1}, 
we only need to show that $u_{\widetilde{T}}\in\mathcal{O}(\C,E)$ because 
then $-[u_{T}]\in\mathcal{B}([0,\infty[,F_{A})$ is a solution of the ACP \eqref{eq:ACP}. 
Now, we repeat the argument from \cite[Theorem 7.6, p.\ 65]{langenbruch2011}. 
For $j\in\N$ and $R\in L(H,E)$ we set 
$\langle \tau_{-j}R, f\rangle\coloneq\langle R, f(\cdot+j)\rangle$ for $f\in H$. Then 
\begin{equation}\label{eq:solution_ACP_5}
\mathscr{L}(\tau_{-j}R)=e^{-j(\cdot)}\mathscr{L}(R)
\end{equation}
by the definition of the Laplace transform $\mathscr{L}$ in \cite[p.\ 133--134]{domanskilangenbruch2010}. It follows 
from \eqref{eq:solution_ACP_2} that $e^{j(\cdot)}\widetilde{s}\in H_{\operatorname{exp}}(\mathcal{E})$ 
and thus there exists $\widetilde{T}_{j}\in L(H,E)$ such that $\mathscr{L}(\widetilde{T}_{j})=e^{j(\cdot)}\widetilde{s}$ 
by \prettyref{thm:laplace_on_operators}, implying 
\[
 \mathscr{L}(\tau_{-j}\widetilde{T}_{j})
\underset{\eqref{eq:solution_ACP_5}}{=}e^{-j(\cdot)}\mathscr{L}(\widetilde{T}_{j})
=\widetilde{s}
=\mathscr{L}(\widetilde{T})
\]
and therefore $\tau_{-j}\widetilde{T}_{j}=\widetilde{T}$ by \prettyref{thm:laplace_on_operators} again. 
We deduce for any $j\in\N$ that 
\[
 u_{\widetilde{T}}(\lambda)
=\langle\tau_{-j}\widetilde{T}_{j},f_{\lambda}\rangle 
=\langle\widetilde{T}_{j},f_{\lambda}(\cdot+j)\rangle 
=\langle\widetilde{T}_{j},f_{\lambda-j}\rangle 
\]
is holomorphic for $\lambda\notin [j,\infty[$ because $\widetilde{T}_{j}\in L(H,E)$, which proves our statement.
\end{proof}

Our next goal is to generalise Langenbruch's sufficient criterion \cite[Theorem 7.7, p.\ 66]{langenbruch2011} 
for the solvability of the ACP \eqref{eq:ACP}, which is done by using a suitable notion of an asymptotic right
resolvent. Let $E$ be a complete $\C$-lcHs and $A\colon F\coloneq D(A)\subset E\to E$ a closed linear operator. 
If $E$ is bornological, i.e.\
\[
E= \lim\limits_{\substack{\longrightarrow\\\mathscr{B}\in \mathfrak{B}^{E}}}\,E_{\mathscr{B}}
\] 
where $\mathfrak{B}^{E}$ is the system of bounded closed absolutely convex subsets of $E$ and 
$E_{\mathscr{B}}\coloneq\operatorname{span}(\mathscr{B})$ equipped with the gauge norm induced 
by $\mathscr{B}\in \mathfrak{B}^{E}$, then the topological identities
\[
 L_{b}(E,E)
=\lim\limits_{\substack{\longleftarrow\\(\mathscr{B},\alpha)\in \mathfrak{B}^{E}\times\mathfrak{A}}}L(E_{\mathscr{B}},E_{\alpha})
 \quad\text{and}\quad
 L_{b}(E,F_{A})
=\lim\limits_{\substack{\longleftarrow\\(\mathscr{B},\alpha)\in \mathfrak{B}^{E}\times\mathfrak{A}}}L(E_{\mathscr{B}},F_{\alpha})
\]
hold by \cite[p.\ 136--137]{domanskilangenbruch2010}. 
This means that the local Banach spaces of $L_{b}(E,E)$ and $L_{b}(E,F_{A})$ are the spaces $L(E_{\mathscr{B}},E_{\alpha})$ 
and $L(E_{\mathscr{B}},F_{\alpha})$ equipped with the operator norm, respectively. 
We set $\mathfrak{C}\coloneq\mathfrak{B}^{E}\times\mathfrak{A}$. A spectral-valued holomorphic operator function 
\[
\mathcal{R}\coloneq (R_{\mathscr{A},\alpha})_{(\mathscr{A},\alpha)\in\mathfrak{C}}\colon 
\mathcal{G}\coloneq (G_{\mathscr{A},\alpha})_{(\mathscr{A},\alpha)\in\mathfrak{C}}
\to \mathcal{L}(E,F_{A})\coloneq (L(E_{\mathscr{A}},F_{\alpha}))_{(\mathscr{A},\alpha)\in\mathfrak{C}}
\]
is called an \emph{asymptotic right resolvent} if $\mathcal{R}\in H_{\operatorname{exp}}(\mathcal{L}(E,F_{A}))$ and 
if there is a spectral-valued holomorphic function
\[
\mathcal{T}\coloneq (T_{\mathscr{A},\alpha})_{(\mathscr{A},\alpha)\in\mathfrak{C}}
\colon \widetilde{\mathcal{G}}\coloneq (\widetilde{G}_{\mathscr{A},\alpha})_{(\mathscr{A},\alpha)\in\mathfrak{C}}
\to\mathcal{L}(E)\coloneq (L(E_{\mathscr{A}},E_{\alpha}))_{(\mathscr{A},\alpha)\in\mathfrak{C}}
\]
such that for any $(\mathscr{A},\alpha)\in\mathfrak{C}$ 
there is $(\mathscr{B},\beta)\in\mathfrak{C}$, $(\mathscr{A},\alpha)\leq(\mathscr{B},\beta)$, such that
\begin{equation}\label{eq:asymp_right_res_1}
 (\lambda I_{\alpha}^{\beta}-A_{\alpha}^{\beta})R_{\mathscr{B},\beta}(\lambda)
={\kappa_{\alpha}^{E}}_{\mid E_{\mathscr{B}}}+T_{\mathscr{B},\alpha}(\lambda), 
 \quad \lambda\in G_{\mathscr{B},\beta}\cap \widetilde{G}_{\mathscr{B},\alpha},
\end{equation}
and for any $j\in\N$ and any $0<\varphi<\tfrac{\pi}{2}$ there is $r>0$ with 
$\Gamma_{r,\varphi}\subset (G_{\mathscr{B},\beta}\cap \widetilde{G}_{\mathscr{B},\alpha})$ and
\begin{equation}\label{eq:asymp_right_res_2}
\sup_{\lambda\in\Gamma_{r,\varphi}}\|T_{B,\alpha}(\lambda)\|_{L(E_{\mathscr{B}},E_{\alpha})}
e^{j\re(\lambda)-\frac{1}{j}|\im(\lambda)|}<\infty.
\end{equation}

\begin{thm}\label{thm:solution_ACP_sufficient}
Let $E$ be an admissible complete bornological $\C$-lcHs and $A\colon F\coloneq D(A)\subset E\to E$ a closed linear 
operator and $F_{A}=(F,\tau_{A})$ admissible. The ACP \eqref{eq:ACP} has a solution 
(in the sense of hyperfunctions) for any $x_{0}\in E$ if $A$ admits an asymptotic right resolvent.
\end{thm}
\begin{proof}
In order to apply \prettyref{thm:solution_ACP} we have to construct a suitable spectral-valued holomorphic function
$\mathcal{S}\coloneq (S_{\alpha})_{\alpha\in\mathfrak{A}}\in H_{\operatorname{exp}}(\mathcal{F})$. 
For $x_{0}\in E$ we choose $\mathscr{A}\in\mathfrak{B}^{E}$ such that $x_{0}\in \mathscr{A}$. 
For $\alpha\in\mathfrak{A}$ we set 
\[
M_{\alpha}\coloneq\{(\mathscr{B},\beta)\in\mathfrak{C}\;|\;(\mathscr{A},\alpha)\leq(\mathscr{B},\beta),\;
\eqref{eq:asymp_right_res_1}\;\text{and}\;\eqref{eq:asymp_right_res_2}\;\text{are satisfied} \}
\] 
and
\[
 G_{\alpha}\coloneq\bigcup_{(\mathscr{B},\beta)\in M_{\alpha}}(G_{\mathscr{B},\beta}\cap 
                                                         \widetilde{G}_{\mathscr{B},\alpha}).
\]
The sets $G_{\alpha}\subset\C$ are non-empty by assumption as well as open, connected and postsectorial 
as they are unions of such sets. 
Next, we show that $M_{\alpha}\subset M_{\gamma}$ for $\alpha,\gamma\in\mathfrak{A}$ with $\gamma\leq\alpha$, 
which then implies $G_{\alpha}\subset G_{\gamma}$ and 
means that $\mathcal{G}_{0}\coloneq (G_{\alpha})_{\alpha\in\mathfrak{A}}$ is directed. 
Let $(\mathscr{B},\beta)\in M_{\alpha}$. Then $(\mathscr{A},\gamma)\leq(\mathscr{B},\beta)$ and we note that 
\begin{equation}\label{thm:solution_ACP_sufficient_0}
\kappa_{\mathscr{B},\gamma}^{\mathscr{B},\alpha}(f)=\kappa_{\gamma}^{\alpha}\circ f
\end{equation}
for all $f\in L(E_{\mathscr{B}},E_{\alpha})$ where $\kappa_{\mathscr{B},\gamma}^{\mathscr{B},\alpha}\colon L(E_{\mathscr{B}},E_{\alpha})\to L(E_{\mathscr{B}},E_{\gamma})$
is the linking map of the local Banach spaces. It holds by \eqref{eq:asymp_right_res_1} and 
the compatibility condition (ii) for $\mathcal{T}$ that
\begin{align}\label{eq:solution_ACP_sufficient_1}
 (\lambda I_{\gamma}^{\beta}-A_{\gamma}^{\beta})R_{\mathscr{B},\beta}(\lambda)
&=\kappa_{\gamma}^{\alpha}(\lambda I_{\alpha}^{\beta}-A_{\alpha}^{\beta})R_{\mathscr{B},\beta}(\lambda)
 =\kappa_{\gamma}^{\alpha}({\kappa_{\alpha}^{E}}_{\mid E_{\mathscr{B}}}+T_{\mathscr{B},\alpha}(\lambda))\notag\\
&={\kappa_{\gamma}^{E}}_{\mid E_{\mathscr{B}}}+\kappa_{\gamma}^{\alpha}\circ T_{\mathscr{B},\alpha}(\lambda)
 \underset{(ii),\eqref{thm:solution_ACP_sufficient_0}}{=}{\kappa_{\gamma}^{E}}_{\mid E_{\mathscr{B}}}
 +T_{\mathscr{B},\gamma}(\lambda)
\end{align}
for all $\lambda\in G_{\mathscr{B},\beta}\cap\widetilde{G}_{\mathscr{B},\alpha}$. 
Since $(G_{\mathscr{B},\beta}\cap\widetilde{G}_{\mathscr{B},\alpha})\subset (G_{\mathscr{B},\beta}
\cap\widetilde{G}_{\mathscr{B},\gamma})$,
the identity theorem implies that \eqref{eq:solution_ACP_sufficient_1} holds on the connected set 
$G_{\mathscr{B},\beta}\cap\widetilde{G}_{\mathscr{B},\gamma}$ as well. Moreover, from the inclusion 
$(G_{\mathscr{B},\beta}\cap\widetilde{G}_{\mathscr{B},\alpha})\subset (G_{\mathscr{B},\beta}
\cap\widetilde{G}_{\mathscr{B},\gamma})$ and 
$T_{\mathscr{B},\gamma}(\lambda)=(\kappa_{\gamma}^{\alpha}\circ T_{\mathscr{B},\alpha})(\lambda)$ for all 
$\lambda\in\widetilde{G}_{\mathscr{B},\alpha}$ it follows that \eqref{eq:asymp_right_res_2} holds with $\alpha$ 
replaced by $\gamma$ too. Hence $(\mathscr{B},\beta)\in M_{\gamma}$, implying $M_{\alpha}\subset M_{\gamma}$.

Now, let $(\mathscr{B}_{1},\beta_{1}),(\mathscr{B}_{2},\beta_{2})\in M_{\alpha}$. 
Then $\emptyset\neq(G_{\mathscr{B}_{1},\beta_{1}}\cap G_{\mathscr{B}_{2},\beta_{2}})
\subset G_{\mathscr{A},\alpha}$ and the compatibility (ii) for $\mathcal{R}$ yields
\[
 \langle \kappa_{\alpha}^{\beta_{i}}\circ R_{\mathscr{B}_{i},\beta_{i}}(\lambda),x_{0}\rangle 
=\langle (\kappa_{\mathscr{A},\alpha}^{\mathscr{B}_{i},\beta_{i}}\circ 
  R_{\mathscr{B}_{i},\beta_{i}})(\lambda),x_{0}\rangle 
\underset{(ii)}{=}R_{\mathscr{A},\alpha}(\lambda)(x_{0})
\]
for all $\lambda\in G_{\mathscr{B}_{i},\beta_{i}}$ and $i=1,2$ where 
$\kappa_{\mathscr{A},\alpha}^{\mathscr{B}_{i},\beta_{i}}\colon 
L(E_{\mathscr{B}_{i}},F_{\beta_{i}})\to L(E_{\mathscr{A}},F_{\alpha})$ is the linking map of 
the local Banach spaces. This implies 
$\langle \kappa_{\alpha}^{\beta_{1}}\circ R_{\mathscr{B}_{1},\beta_{1}}(\lambda),x_{0}\rangle 
=\langle \kappa_{\alpha}^{\beta_{2}}\circ R_{\mathscr{B}_{2},\beta_{2}}(\lambda),x_{0}\rangle$ for all 
$\lambda\in G_{\mathscr{B}_{1},\beta_{1}}\cap G_{\mathscr{B}_{2},\beta_{2}}$. 
Therefore the map $S_{\alpha}\colon G_{\alpha}\to F_{\alpha}$ given by 
$S_{\alpha}(\lambda)\coloneq\langle\kappa_{\alpha}^{\beta}\circ R_{\mathscr{B},\beta}(\lambda),x_{0}\rangle$ on 
$G_{\mathscr{B},\beta}\cap\widetilde{G}_{\mathscr{B},\alpha}$ for 
$(\mathscr{B},\beta)\in M_{\alpha}$ is well-defined and holomorphic on $G_{\alpha}$. 
This gives us (i) for $S\coloneq (S_{\alpha})_{\alpha\in\mathfrak{A}}\colon \mathcal{G}_{0}\to\mathcal{F}$.

Let us turn to the compatibility condition (ii) for $S$. Let $\alpha,\gamma\in\mathfrak{A}$ with $\alpha\leq\gamma$. 
Then for any $(\mathscr{B},\beta)\in M_{\gamma}\subset M_{\alpha}$ and 
$\lambda\in G_{\mathscr{B},\beta}\cap\widetilde{G}_{\mathscr{B},\alpha}$ we have 
\[
 (\kappa_{\alpha}^{\gamma}\circ S_{\gamma})(\lambda)
=\langle\kappa_{\alpha}^{\gamma}\circ \kappa_{\gamma}^{\beta}\circ R_{\mathscr{B},\beta}(\lambda),x_{0}\rangle
=\langle\kappa_{\alpha}^{\beta}\circ R_{\mathscr{B},\beta}(\lambda),x_{0}\rangle
=S_{\alpha}(\lambda).
\]
We derive that $S$ fulfils (ii) and is an $\mathcal{F}$-valued holomorphic function.

Since $\mathcal{R}\in H_{\operatorname{exp}}(\mathcal{L}(E,F_{A}))$, for $(\mathscr{B},\beta)\in M_{\alpha}$ 
and any $K\in\N$ and any $0<\varphi<\tfrac{\pi}{2}$, there is $r>0$ such that 
$\Gamma_{r,\varphi}\subset G_{\mathscr{B},\beta}$ and 
\[
\sup_{\lambda\in\Gamma_{r,\varphi}}\|R_{\mathscr{B},\beta}(\lambda)\|_{L(E_{\mathscr{B}},F_{\beta})}
e^{-\frac{1}{K}\re(\lambda)}<\infty.
\]
The set $G_{\mathscr{B},\beta}\cap\widetilde{G}_{\mathscr{B},\alpha}$ is postsectorial and so there is $t\geq r$ with 
$\Gamma_{t,\varphi}\subset \Gamma_{r,\varphi}$ and 
$\Gamma_{t,\varphi}\subset (G_{\mathscr{B},\beta}\cap\widetilde{G}_{\mathscr{B},\alpha})$. 
We remark that the continuity of $\kappa_{\alpha}^{\beta}$ implies that there is $C_{1}>0$ such that 
\begin{align*}
 \|S_{\alpha}(\lambda)\|_{\alpha}
&=\|\langle\kappa_{\alpha}^{\beta}\circ R_{\mathscr{B},\beta}(\lambda),x_{0}\rangle\|_{\alpha}
 \leq C_{1}\|R_{\mathscr{B},\beta}(\lambda)x_{0}\|_{\beta,A}\\
&\leq C_{1}\|R_{\mathscr{B},\beta}(\lambda)\|_{L(E_{\mathscr{B}},F_{\beta})}\|x_{0}\|_{E_{\mathscr{B}}}
\end{align*}
for all $\lambda\in G_{\mathscr{B},\beta}\cap\widetilde{G}_{\mathscr{B},\alpha}$. It follows that 
\[
     \sup_{\lambda\in\Gamma_{t,\varphi}} \|S_{\alpha}(\lambda)\|_{\alpha}e^{-\frac{1}{K}\re(\lambda)}
\leq C_{1}\|x_{0}\|_{E_{\mathscr{B}}}\sup_{\lambda\in\Gamma_{r,\varphi}}
     \|R_{\mathscr{B},\beta}(\lambda)\|_{L(E_{\mathscr{B}},F_{\beta})}e^{-\frac{1}{K}\re(\lambda)}<\infty
\]
and we conclude that $S\in H_{\operatorname{exp}}(\mathcal{F})$.

We define $s_{\alpha}^{\beta}\colon G_{\beta}\to E_{\alpha}$ for $\beta$ with $(\mathscr{B},\beta)\in M_{\alpha}$ 
by \eqref{eq:solution_ACP_1} as before and note that for $(\mathscr{B}_{1},\gamma)\in M_{\beta}\subset M_{\alpha}$
\begin{align*}
 s_{\alpha}^{\beta}(\lambda)
&=(\lambda I_{\alpha}^{\beta}-A_{\alpha}^{\beta})S_{\beta}(\lambda)-\kappa_{\alpha}^{E}(x_{0})
 =(\lambda I_{\alpha}^{\beta}-A_{\alpha}^{\beta})\kappa_{\beta}^{\gamma}R_{\mathscr{B}_{1},\gamma}(\lambda)x_{0}
 -\kappa_{\alpha}^{E}(x_{0})\\
&=(\lambda I_{\alpha}^{\gamma}-A_{\alpha}^{\gamma})R_{\mathscr{B}_{1},\gamma}(\lambda)x_{0}
 -\kappa_{\alpha}^{E}(x_{0})
\underset{\eqref{eq:asymp_right_res_1}}{=}\kappa_{\alpha}^{E}(x_{0})+T_{\mathscr{B}_{1},\alpha}(\lambda)x_{0}
 -\kappa_{\alpha}^{E}(x_{0})\\
&=T_{\mathscr{B}_{1},\alpha}(\lambda)x_{0}
\end{align*}
for all $\lambda\in (G_{\mathscr{B}_{1},\gamma}\cap\widetilde{G}_{\mathscr{B}_{1},\beta})\subset
(G_{\mathscr{B}_{1},\gamma}\cap\widetilde{G}_{\mathscr{B}_{1},\alpha})$. As $(\mathscr{B}_{1},\gamma)\in M_{\beta}$, 
for any $j\in\N$ and any $0<\varphi<\tfrac{\pi}{2}$ there is $r>0$ with 
$\Gamma_{r,\varphi}\subset (G_{\mathscr{B}_{1},\gamma}\cap \widetilde{G}_{\mathscr{B}_{1},\beta})$ and
\begin{equation}\label{eq:solution_ACP_sufficient_2}
\sup_{\lambda\in\Gamma_{r,\varphi}}\|T_{\mathscr{B}_{1},\beta}(\lambda)\|_{L(E_{\mathscr{B}_{1}},E_{\beta})}
e^{j\re(\lambda)-\frac{1}{j}|\im(\lambda)|}<\infty.
\end{equation}
From the compatibility condition (ii) of $\mathcal{T}$ we deduce that 
\[
 s_{\alpha}^{\beta}(\lambda)
=T_{\mathscr{B}_{1},\alpha}(\lambda)x_{0}
=\langle(\kappa_{\mathscr{B}_{1},\alpha}^{\mathscr{B}_{1},\beta}
 \circ T_{\mathscr{B}_{1},\beta})(\lambda),x_{0}\rangle
\]
for all $\lambda\in G_{\mathscr{B}_{1},\gamma}\cap\widetilde{G}_{\mathscr{B}_{1},\beta}$ where 
$\kappa_{\mathscr{B}_{1},\alpha}^{\mathscr{B}_{1},\beta}\colon L(E_{\mathscr{B}_{1}},E_{\beta})\to 
L(E_{\mathscr{B}_{1}},E_{\alpha})$ is the linking map of the local Banach spaces. 
The continuity of the linking map implies that there is $C_{2}>0$ such that 
\begin{align*}
 \|s_{\alpha}^{\beta}(\lambda)\|_{\alpha}
&=\|\langle(\kappa_{\mathscr{B}_{1},\alpha}^{\mathscr{B}_{1},\beta}
  \circ T_{\mathscr{B}_{1},\beta})(\lambda),x_{0}\rangle\|_{\alpha}
 \leq C_{2}\|T_{\mathscr{B}_{1},\beta}(\lambda)x_{0}\|_{\beta}\\
&\leq C_{2}\|T_{\mathscr{B}_{1},\beta}(\lambda)\|_{L(E_{\mathscr{B}_{1}},E_{\beta})}\|x_{0}\|_{E_{\mathscr{B}_{1}}}
\end{align*}
for all $\lambda\in G_{\mathscr{B}_{1},\gamma}\cap\widetilde{G}_{\mathscr{B}_{1},\beta}$. In combination with 
\eqref{eq:solution_ACP_sufficient_2} we get \eqref{eq:solution_ACP_2}. 
Applying \prettyref{thm:solution_ACP}, we obtain our statement.
\end{proof}

We illustrate \prettyref{thm:solution_ACP} by an application to the one-dimensional heat equation in the space 
of tempered distributions.
Let $\mathcal{S}(\R)$ be the Schwartz space, i.e.\
\[
\mathcal{S}(\R)
\coloneq\{f\in \mathcal{C}^{\infty}(\R)\;|\;
\forall\;n\in\N_{0}:\;|f|_{n}^{\mathcal{S}(\R)}<\infty\}
\]
where 
\[
 |f|_{n}^{\mathcal{S}(\R)}\coloneq\sup_{\substack{x\in\R\\ m\in\N_{0},m\leq n}}
 |f^{(m)}(x)|(1+|x|^{2})^{n/2}.
\]
Further, we equip the space $\mathcal{C}^{\infty}(\R)$ with its usual topology of uniform convergence of 
partial derivatives up to any order on compact subsets of $\R$. 

\begin{thm}
Let $x_{0}\in\mathcal{C}^{\infty}(\R)'$. Then the ACP
\[
x'(t)=\Delta x(t),\quad t>0, \quad x(0)=x_{0},
\]
has a solution $x\in \mathcal{B}([0,\infty[,\mathcal{S}(\R)_{b}')$ in the sense of hyperfunctions. 
\end{thm}
\begin{proof}
We set $f\colon \C_{\re >0}\times\R\to \C$, 
$f(\lambda,s)\coloneq\tfrac{1}{2\sqrt{\lambda}}e^{-\sqrt{\lambda}|s|}$, where $\sqrt{\cdot}$ is the principal square root, 
i.e.\ $\sqrt{\lambda}=\sqrt{|\lambda|}(\cos(\tfrac{\operatorname{arg}(\lambda)}{2})
+i\sin(\tfrac{\operatorname{arg}(\lambda)}{2}))$ with the principal argument 
$\operatorname{arg}(\lambda)\in (-\tfrac{\pi}{2},\tfrac{\pi}{2})$ for $\lambda\in \C_{\re >0}$.
Then $f(\lambda,\cdot)$ is continuous and thus Borel-measurable for all $\lambda\in\C_{\re >0}$ and 
\begin{equation}\label{eq:integrable_heat_eq}
 \int_{\R}|f(\lambda,s)|\d s
=\frac{1}{2\sqrt{|\lambda|}}\int_{\R}e^{-\sqrt{|\lambda|}\cos(\frac{\operatorname{arg}(\lambda)}{2})|s|}\d s
 \leq \frac{1}{2\sqrt{|\lambda|}}\int_{\R}e^{-\sqrt{\frac{|\lambda|}{2}}|s|}\d s
 =\frac{\sqrt{2}}{|\lambda|},
\end{equation}
which means that $f(\lambda,\cdot)\in L^{1}(\R,\C)$ for all $\lambda\in\C_{\re >0}$. Therefore the distributional 
convolution $f(\lambda,\cdot)\ast x_{0}\in\mathcal{S}(\R)'$ for all $\lambda\in\C_{\re >0}$ 
by \cite[Theorem 7.1.15, p.\ 166]{H1} where 
\[
\langle f(\lambda,\cdot)\ast x_{0},\psi\rangle\coloneq\langle f(\lambda,\cdot) ,\check{x}_{0}\ast\psi\rangle 
=\int_{\R}f(\lambda,s)(\check{x}_{0}\ast\psi)(s)\d s
\] 
and 
\[
(\check{x}_{0}\ast\psi)(s)\coloneq\langle x_{0},\psi(s+\cdot)\rangle,\;s\in\R,
\]
for $\psi\in\mathcal{S}(\R)$. 
The map $\lambda\mapsto\langle f(\lambda,\cdot)\ast x_{0},\psi\rangle$ is holomorphic for all 
$\psi\in\mathcal{S}(\R)$ by differentiation under the integral w.r.t.\ the parameter $\lambda$ 
due to \eqref{eq:integrable_heat_eq} and \cite[5.8 Satz, p.\ 148--149]{elstrodt2005} with the majorant 
$g_{K}(s)\coloneq\tfrac{1}{2\sqrt{C_{K}}}e^{-\sqrt{\frac{C_{K}}{2}}|s|}|(\check{x}_{0}\ast\psi)(s)|$, $s\in\R$, where 
$C_{K}\coloneq\min_{\lambda\in K}|\lambda|$ for any compact disc $K\subset\C_{\re >0}$. As $\mathcal{S}(\R)$ is reflexive 
and $\mathcal{S}(\R)_{b}'$ complete, this means that $\lambda\mapsto f(\lambda,\cdot)\ast x_{0}$ 
is weakly holomorpic and thus holomorphic, 
i.e.\ $(\lambda\mapsto f(\lambda,\cdot)\ast x_{0})\in\mathcal{O}(\C_{\re >0},\mathcal{S}(\R)_{b}')$, by 
\cite[16.7.2 Theorem, p.\ 362--363]{Jarchow}. 
Since $x_{0}\in \mathcal{C}^{\infty}(\R)'$, there are $C_{0}\geq 0$ and $n\in\N_{0}$ such that for all 
$\lambda\in\C_{\re >0}$ and $\psi\in\mathcal{S}(\R)$ we have
\begin{align*}
 |\langle f(\lambda,\cdot)\ast x_{0},\psi\rangle|
&=\bigl|\int_{\R}f(\lambda,s)(\check{x}_{0}\ast\psi)(s)\d s\bigr|\\
&\leq \frac{\sqrt{2}}{|\lambda|}\sup_{s\in\R}|(\check{x}_{0}\ast\psi)(s)|
 \leq \frac{\sqrt{2}C_{0}}{|\lambda|}\sup_{s\in\R}\sup_{\substack{y\in [-n,n]\\m\leq n}}|\psi^{(m)}(s+y)|\\
&\leq \frac{\sqrt{2}C_{0}}{|\lambda|}\sup_{\substack{s\in\R\\m\leq n}}|\psi^{(m)}(s)|(1+|s|^{2})^{\tfrac{n}{2}}
 =\frac{\sqrt{2}C_{0}}{|\lambda|}|\psi|_{n}^{\mathcal{S}(\R)},
\end{align*}
implying for every bounded set $\beta\subset\mathcal{S}(\R)$ that 
\[
 \sup_{\psi\in \beta}|\langle f(\lambda,\cdot)\ast x_{0},\psi\rangle|
\leq \frac{\sqrt{2}C_{0}}{|\lambda|}\sup_{\psi\in \beta}|\psi|_{n}^{\mathcal{S}(\R)}
=\frac{\sqrt{2}C_{0}C_{n,\beta}}{|\lambda|}
\]
with $C_{n,\beta}\coloneq\sup_{\psi\in \beta}|\psi|_{n}^{\mathcal{S}}<\infty$. We deduce that 
for all bounded sets $\beta\subset\mathcal{S}(\R)$, $K\in\N$, $0<\varphi<\tfrac{\pi}{2}$ and any $r\geq 1$ 
it holds that $\Gamma_{r,\varphi}\subset \C_{\re >0}$ and 
\begin{align}\label{eq:sector_heat_eq}
 \sup_{\lambda\in\Gamma_{r,\varphi}}\sup_{\psi\in \beta}|\langle f(\lambda,\cdot)\ast x_{0},\psi\rangle|
    e^{-\frac{1}{K}\re(\lambda)}
&\leq\sup_{\lambda\in\Gamma_{r,\varphi}}\sup_{\psi\in \beta}|\langle f(\lambda,\cdot)\ast x_{0},\psi\rangle||\lambda|
 \notag\\
&\leq\sqrt{2}C_{0}C_{n,\beta}<\infty.
\end{align}
Furthermore, we have
\begin{equation}\label{eq:resolvent_heat_eq}
 (\lambda-\Delta)(f(\lambda,\cdot)\ast x_{0})
=((\lambda-\Delta)f(\lambda,\cdot))\ast x_{0}
=\delta_{s=0}\ast x_{0}=x_{0}
\end{equation}
for all $\lambda\in\C_{\re >0}$ (cf.\ \cite[p.\ 249--251]{Kom6}). 
The complete space $E\coloneq F\coloneq\mathcal{S}(\R)_{b}'$ is admissible by \cite[Example 4.4, p.\ 14--15]{kruse2019_5},
$\Delta\colon E\to F$ continuous, in particular $F=F_{\Delta}$, and $x_{0}\in\mathcal{C}^{\infty}(\R)'
\subset\mathcal{S}(\R)'$ as linear spaces. Setting  
$S_{\beta}(\lambda)\coloneq\kappa_{\beta}^{E}(f(\lambda,\cdot)\ast x_{0})$ for $\lambda\in G_{\beta}\coloneq\C_{\re >0}$ and 
\[
s_{\alpha}^{\beta}(\lambda)\coloneq (\lambda I_{\alpha}^{\beta}-\Delta_{\alpha}^{\beta})S_{\beta}(\lambda)
 -\kappa_{\alpha}^{E}(x_{0})\underset{\eqref{eq:resolvent_heat_eq}}{=}0,\quad \lambda\in G_{\beta},
\] 
for bounded sets $\alpha,\beta\subset\mathcal{S}(\R)$ with $\alpha\subset\beta$, 
we conclude the existence of a solution $x\in \mathcal{B}([0,\infty[,\mathcal{S}(\R)_{b}')$ 
in the sense of hyperfunctions from \prettyref{thm:solution_ACP} and \eqref{eq:sector_heat_eq}.
\end{proof}

\bibliography{biblio}
\bibliographystyle{plainnat}
\end{document}